\documentclass[11pt,a4paper]{article}

\usepackage[utf8]{inputenc}
\usepackage[cyr]{aeguill}
\usepackage[english,francais]{babel}
\usepackage{amsmath}
\usepackage{amsthm}
\usepackage{amsfonts}
\usepackage{amssymb}
\usepackage{amsbsy}
\usepackage{bbm}
\usepackage{mathrsfs}
\usepackage{enumerate}
\usepackage{setspace}
\usepackage{multicol}

\usepackage{graphicx, psfrag, pstricks}
\usepackage{color}
\usepackage{url}

\def \eps{\varepsilon}
\def \a{\alpha}

\def \g{\gamma}
\def \s{\sigma}
\def \d{\delta}
\def \l{\lambda}
\def \k{\xi}
\def \vp{\varphi}
\def \G{\Gamma}

\def \R{\mathbb R}
\def \C{\mathbb C}
\def \Z{\mathbb Z}
\def \N{\mathbb N}
\def \E{\mathbb E}
\def \K{\mathbb K}

\def \supp{\text{supp}}
\def \su{supp \,}
\def \rep{\text{rep}}

\def \A{\mathcal A}

\def \P{\mathcal P}

\def \H{\mathcal H}
\def \F{\mathcal F}
\def \cara{\mathbbm 1}
\def \Vect{ {\rm Vect}}
\def \st{\; | \;}
\def \di{\varDelta}
\def \ds{\bigtriangleup}

\def \bu{\bullet}
\newcommand*{\longhookrightarrow}{\ensuremath{\lhook\joinrel\relbar\joinrel\rightarrow}}

\newcommand{\lexp}[2]{{\vphantom{#2}}^{#1}\hspace{-0.8mm}#2}
\newcommand{\rexp}[2]{#2\hspace{-0.7mm}{\vphantom{}}^{#1}}

\newcommand{\closeunder}[2]{\underset{\raise\hspace{-0.8mm}box{1mm}{\ensuremath{#2}}}{#1}}
\newcommand*{\quot}[2]%
{\ensuremath{%
    #1/\!\raisebox{-.65ex}{\ensuremath{#2}}}}
\newtheorem{theor}{Theorem}
\newtheorem{corr}[theor]{Corollary}
\newtheorem{theo}{Theorem}[section]
\newtheorem{prop}[theo]{Proposition}

\newtheorem{dfpr}[theo]{Definition-Proposition}
\newtheorem{nt}[theo]{Notation}
\newtheorem{lem}[theo]{Lemma}

\newtheorem{cor}[theo]{Corollary}
\newtheorem{df}[theo]{Definition}
\newtheorem{rmq}[theo]{Remark}
\newtheorem{exemple}[theo]{Example}

\newenvironment{noth}[1][]%
	{\par\parindent0pt\medskip\textbf{Theorem\ifx&#1&\else\space#1\fi.}%
	  \itshape\par}{\par\medskip}
	 
\newenvironment{nocor}[1][]%
	{\par\parindent0pt\medskip\textbf{Corollary\ifx&#1&\else\space#1\fi.}%
	  \itshape\par}{\par\medskip}

\title{Isometric affine actions on Banach spaces and spaces with labelled partitions}
\date{}
\author{S. Arnt}

\begin{document}

\renewcommand{\proofname}{Proof}
\renewcommand\refname{References}
\renewcommand\contentsname{Table of contents}
\renewcommand{\abstractname}{Abstract}
\renewcommand{\thefootnote}{\*}

\maketitle

\footnotetext{The research of the author was partially supported by grant 20-149261 of Swiss SNF.}

\begin{abstract}
We define the structure of spaces with labelled partitions which generalizes the structure of spaces with measured walls and study the link between actions by automorphisms on spaces with labelled partitions and isometric affine actions on Banach spaces, and more particularly, on $L^p$ spaces. We build natural spaces with labelled partitions for the action of various contructions of groups, namely: direct sum; semi-direct product; wreath product and amalgamated free product. We apply this to prove that the wreath product of a group with property $PL^p$ by a group with Haagerup property has property $PL^p$ and the amalgamated free product of groups with property $PL^p$ has property $PL^p$.
\end{abstract}

\section{Introduction}
A locally compact second countable group $G$ has \emph{Haagerup property} (or is \emph{a-(T)-menable}) if there exists a proper continuous isometric affine action of $G$ on a Hilbert space; this property can be seen as a strong negation of Kazhdan's property (T) (an overview of the Haagerup property can be found in \cite{checowjoljulval}). Groups having Haagerup property are known to satisfy the Baum-Connes conjecture by a result of Higson and Kasparov in \cite{higkas} (see \cite{julg} for further details).
Haagerup property is closed by taking subgroups, direct products, amalgamated products over finite subsets but it is not stable by group extensions in general, even in the case of semi-direct products. However, Cornulier, Stalder and Valette recently proved in \cite{wreath} that it is stable by a particular kind of extension, namely the wreath product. They use for their proof the connexion between Haagerup property and spaces with measured walls, that we will now explain.

A \emph{space with walls} is a pair $(X,W)$ where $X$ is a set and $W$ is a family of partitions of $X$ in two pieces called \emph{walls} such that any pair of points of $X$ is separated by finitely many walls. This notion was introduced by Haglund and Paulin in \cite{hagpaul} and generalized in a topological setting by Cherix, Martin and Valette in \cite{chemarval} to \emph{space with measured walls} (see Definition \ref{labpart_meswalls}). It was gradually realised that the Haagerup property is equivalent to the existence of a proper action on a space with measured walls; more precisely, we have the following theorem: \emph{a locally compact second countable group has the Haagerup property if, and only if, it acts properly by automorphisms on a space with measured walls.} Using results of Robertson and Steger (see \cite{robste}), Cherix, Martin and Valette in \cite{chemarval}, proved this theorem for discrete groups and Chatterji, Drutu and Haglund extended the equivalence to locally compact second 
countable groups using the notion of median metric spaces in \cite{chadruhag}. The stability of the Haagerup property by wreath product was established in \cite{wreath} by constructing a space with measured walls from the structures of measured walls on each factor, and moreover, in the same article, Cornulier, Stalder and Valette generalized this result to the permutational wreath product (see Definition \ref{wreath_df}) when the index set $I$ is a quotient by a co-Haagerup subgroup of the shifting group $G$ (see \cite{chiioa} for a counter example when the pair $(G,I)$ has relative property (T)). This result led to the first example of Haagerup groups which are not weakly amenable in the sense of \cite{cowhaa}. \\
The notion of Haagerup property naturally extends to proper isometric affine action on Banach spaces. Recent works have been made about isometric actions on Banach spaces: in \cite{haaprz}, Haagerup and Przybyszewska showed that every locally compact second countable group $G$ acts properly by affine isometries on the reflexive Banach space $\bigoplus_{n \in \N}^2 L^{2n}(G,\mu)$ where $\mu$ is the Haar measure; Cornulier, Tessera, Valette introduced in \cite{cortesval} property $(BP_0^V)$ for $V$ a Banach space as a tool to show that the simple Lie group $G=Sp(n,1)$ acts properly by isometries on $L^p(G)$ for $p > 4n+2$; in \cite{badfurgelmon}, Bader, Furman Gelander and Monod, studied an analog of property (T) in terms of $L^p$ spaces and more generally, of superreflexive Banach spaces. One of the motivation of this topic is given by a recent result of Kasparov and Yu in \cite{kasyu} which asserts that the existence of coarse embeddings of a finitely generated group in a uniformly convex Banach space 
implies the coarse geometric Novikov conjecture for this group. See \cite{now} for an overview of results and questions about isometric affine actions on Banach spaces. \\
We will focus on specific Banach spaces, namely, $L^p$ spaces. For $p \geq 1$, we say that a locally compact second countable group $G$ has \emph{property $PL^p$} (or is \emph{a-$FL^p$-menable}) if there exists a proper continuous isometric affine action on a $L^p$ space. See for instance \cite{chadruhag}, for a characterisation of property $PL^p$ for $p\in [1,2]$ in terms of Haagerup property. An important example is the following theorem due to Yu (see \cite{yuhyp}): \emph{let $\G$ be a discrete Gromov hyperbolic group. Then there exists $p \geq 2$ such that $\G$ has property $PL^p$.} Yu proved this result by giving an explicit proper isometric affine action of $\G$ on $\ell^{p}(\G\times \G\st d(x,y)\leq R)$ using a construction of Mineyev in \cite{minhyp}; see \cite{bourd1} or \cite{nicahyp} for other proofs of this result in terms of boundaries of $G$. A remarkable consequence is that there exists infinite groups with property (T) (and hence, without Haagerup property) which have property $PL^p$ for some 
$p >2$.

In this paper, we define a generalization of the structure of spaces with measured walls, namely, the structure of spaces with labelled partitions which provides a flexible framework, in terms of geometry and stability by various type of group constructions, for isometric affine actions on Banach spaces. In Paragraph \ref{subsec_action}, we establish the following result which links isometric affine actions on Banach spaces and actions by automorphisms on spaces with labelled partitions :

\begin{theor}\label{labpart_affact}
   Let $G$ be topological group.
   \begin{enumerate}
    \item[1.] If $G$ acts (resp. acts properly) continuously by affine isometries on a Banach space $B$ then there exists a structure $(G,\P,F(\P))$ of space with labelled partitions on $G$ such that $G$ acts (resp. acts properly) continuously by automorphisms on $(G,\P,F(\P))$ via its left-action on itself. Moreover, there exists a linear isometric embedding $F(\P) \hookrightarrow B$.
    \item[2.] If $G$ acts (resp. acts properly) continuously by automorphisms on a space with labelled partitions $(X,\P,F(\P))$ then there exists a (resp. proper) continuous isometric affine action of $G$ on a Banach space $B$. Moreover, $B$ is a closed subspace of $F(\P)$.
   \end{enumerate}
  \end{theor}
  This theorem can be rephrased in the particular case of $L^p$ spaces as follows:
 \begin{corr}\label{labpart_aflp}
  Let $p\geq 1$ with $p \notin 2\Z\smallsetminus \{2\}$ and $G$ be a topological group. $G$ has property $PL^p$ if, and only if, $G$ acts properly continuously by automorphisms on a space with labelled partitions $(X,\P,F(\P))$ where $F(\P)$ is isometrically isomorphic to a closed subspace of an $L^p$ space. 
 \end{corr}
A crucial factor in the definition of space with labelled partitions is the ``geometric'' understanding of the construction of Mineyev in \cite{minhyp} used by Yu in \cite{yuhyp} to exhibit a proper action of discrete hyperbolic groups on some $\ell^p$ space. Moreover, another inspiration for this definition comes from \cite{checowjoljulval} Proposition 7.4.1 where Valette states the following geometric characterisation of the Haagerup property for locally compact groups : $G$ has the Haagerup property if, and only if, there exists a metric space $(X,d)$ on which $G$ acts isometrically and metrically properly, a unitary representation $\pi$ of $G$ on a Hilbert space $\H_{\pi}$, and a continuous map $c: X \times X \rightarrow \H_{\pi}$ such that :
\begin{itemize}
 \item[1.] Chasles' relation :
 \begin{center}for all $x,y,z \in X$, $c(x,z)=c(x,y)+c(y,z)$;\end{center}
 \item[2.] $G$-equivariance condition :
 \begin{center}for all $x,y \in X$, $g\in G$, $c(gx,gy)=\pi(g)c(x,y)$;\end{center}
 \item[3.] Properness condition :
 \begin{center}if $d(x,y) \rightarrow +\infty$, then $\|c(x,y)\|_{\H_{\pi}}\rightarrow +\infty$.\end{center}
\end{itemize}
To emphasize the connection with this result, we use the same notation $c$ (for \emph{c}ocycle) for the separation map $c:X\times X \rightarrow F(\P)$ associated with a set of labelling functions $\P$ (see Definition \ref{labpart_sepmap}). In fact, an immediate consequence of Theorem \ref{labpart_affact} Statement 1. is that the separation map $c_{\a}$ of the set of labelled partitions associated with a proper isometric affine action $\a$ on a Banach space (see Definition \ref{labpart_alpha_def}), satisfies the conditions 1., 2., and 3. mentionned above. \\

We describe, in Part \ref{subsubsec_defactlab}, the maps that preserve the structure of spaces with labelled partitions in order to define actions by automorphisms on a space with labelled partitions. This notion of homomorphisms of spaces with labelled partitions generalizes the notion of homomorphisms of spaces with measured walls (see \cite{chadruhag} Definition 3.5). \\
We discuss some constructions of spaces with labelled partitions for the direct sum, semi-direct product, wreath product and amalgamated free product in Sections \ref{sec_dirsum}, \ref{sec_wreath} and \ref{freepartsec}. We apply these constructions in the case of groups with property $PL^p$ and we obtain the following stability properties :

\begin{theor}\label{const_semidir_wreathflp}
  Let $H,G$ be countable discrete groups, $L$ be a subgroup of $G$ and $p > 1$, with $p \notin 2\Z\smallsetminus \{2\}$. We denote by $I$ the quotient $G/L$ and $W=\bigoplus_{I}H$. Assume that $G$ is Haagerup, $L$ is co-Haagerup in $G$ and $H$ has property $PL^p$. \\
  Then the permutational wreath product $H\wr_I G=W\rtimes G$ has property $PL^p$.
 \end{theor}
 
 \begin{theor}\label{amalgamflp}
  Let $G,H$ be countable discrete groups, $F$ be a finite group such that $F \hookrightarrow G$ and $F \hookrightarrow H$ and $p > 1$ with $p \notin 2\Z \smallsetminus \{2\}$. Then $G,H$ have property $PL^p$ if, and only if, $G*_F H$ has property $PL^p$.
 \end{theor}
 
 We mention that, in his thesis, Pillon obtains by other methods the stability of property $PL^p$ by amalgamated free product over finite subgroups and he also computes a lower bound for the equivariant $L^p$-compression of such a product in terms of the equivariant $L^p$-compression of the factors. More precisely, he exhibits an explicit proper cocycle for a representation of the product built from the induced representations of the factors.
 
 \newpage

 \tableofcontents
 
 \newpage

Subsequently, all topological groups we consider are assumed to be Hausdorff.

\section{Preliminaries}\label{sec_prelim}

\subsection{Metrically proper actions}\label{sub_prelim_proper}

A pseudo-metric $d$ on a set $X$ is a symmetric map $d: X\times X \rightarrow \R_+$ which satisfies the triangle inequality and $d(x,x)=0$. But unlike a metric, a pseudo-metric need not separate points. 

\begin{df}
 Let $G$ be a topological group acting continuously isometrically on a pseudo-metric space $(X,d_X)$. The $G$-action on $X$ is said \emph{metrically proper} if, for all (or equivalently, for some) $x_0 \in X$ : 
 $$\displaystyle \lim_{g \rightarrow \infty}d_X(g.x_0,x_0)= +\infty ,$$
 i.e., in other words, for all $R \geq 0$, the set $\{ g \in G \st d_X(g.x_0,x_0) \leq R \}$ is relatively compact in $G$.
\end{df}

Let $X$ be a set endowed with a pseudo-metric $d$. We put on $X$ the following equivalence relation: for $x,x' \in X$, $x\sim x'$ if, and only if, $d(x,x')=0$, and we denote by $Y$ the quotient set $X/ \sim$. Then we can define a metric $\tilde{d}$ on $Y$ by setting, for $x,x' \in X$, $\tilde{d}([x],[x'])=d(x,x')$. Moreover, an isometric group action $(X,d)$ preserves the classes of $\sim$ and then induces an isometric action on $(Y,\tilde{d})$.

\begin{lem}
  Let $G$ be a topological group acting continuously isometrically on a pseudo-metric space $(X,d)$. The $G$-action on $X$ is \emph{metrically proper} if, and only if, the induced $G$-action on the quotient metric space $(Y,\tilde{d})$ is metrically proper.
\end{lem}

\subsection{Isometric affine actions}\label{sub_prelim_affact}

\begin{df}
 We say that the action of a topological group $G$ on a topological space $X$ is \emph{strongly continuous} if, for all $x \in X$, the orbit map from $G$ to $X$, $g \mapsto gx$ is continuous.
\end{df}

Let $G$ be a topological group and let $(B,\|.\|)$ be a Banach space on $\K=\R$ or $\C$.

\begin{df}
 A \emph{continuous isometric affine} action $\a$ of $G$ on $B$ is a strongly continuous morphism 
 $$\a: G \longrightarrow \text{Isom}(B)\cap \text{Aff}(B).$$
\end{df}
Notice that if $B$ is a real Banach space, then, by Mazur-Ulam Theorem, $$\text{Isom}(B)\cap \text{Aff}(B)=\text{Isom}(B).$$

\begin{prop}
 A continuous isometric affine action $\a$ of $G$ on $B$ is characterised by a pair $(\pi,b)$ where: 
 \begin{itemize}
  \item $\pi$ is a strongly continuous isometric representation of $G$ on $B$,
  \item $b:G\rightarrow B$ is a continuous map satisfying the 1-cocycle relation: for $g,h \in G$, $$ b(gh)=\pi(g)b(h)+b(g). $$
 \end{itemize}
 And we have, for $g \in G$, $x \in B$:
 $$ \a(g)x=\pi(g)x+b(g). $$
\end{prop}

\begin{df}
 Let $\a$ be a continuous isometric affine action of $G$ on $B$. We say that $\a$ is \emph{proper} if the action of $G$ on the metric space $(B,d_{\|.\|})$ is metrically proper where $d_{\|.\|}$ is the canonical metric on $B$ induced by the norm $\|.\|$.
\end{df}

\begin{prop}
 A continuous isometric affine action $\a$ of $G$ on $B$ is proper if, and only if $$\|b(g)\|\underset{g \rightarrow \infty}{\longrightarrow}+\infty.$$
\end{prop}

\begin{df}
 Let $p \geq 1$. We say that $G$ has property $PL^p$ (or is \emph{a-$FL^p$-menable}) if there exists a proper continuous isometric affine action of $G$ on a $L^p$ space.
\end{df}

\subsection{On isometries of $L^p$-spaces}\label{sub_prelim_lp}

In general, for $p\geq 1$, a closed subspace of a $L^p$-space is not a $L^p$-space (exempt the special case $p=2$); but, in \cite{hardin}, Hardin showed the following result about extension of linear isometries on closed subspace of a $L^p$ (here, we give a reformulation of this result coming from \cite{badfurgelmon}, Corollary 2.20):

\begin{theo}\label{hardin_lp}
 Let $p > 1$ with $p \notin 2\Z\smallsetminus \{2\}$ and $F$ be a closed subspace of $L^p(X,\mu)$. Let $\pi$ be a linear isometric representation of a group $G$ on $F$. Then there is a linear isometric representation $\a'$ of $G$ on some other space $L^p(X',\mu')$ and a linear $G$-equivariant isometric embedding $F \hookrightarrow L^p(X',\mu')$.
\end{theo}

An immediate consequence is the following :

\begin{cor}\label{hardin_lp_cor}
 Let $p > 1$ with $p \notin 2\Z\smallsetminus \{2\}$, $F$ be a closed subspace of a $L^p$-space and $G$ be a topological group. If $G$ acts properly by affine isometries on $F$, then $G$ has property $PL^p$.
\end{cor}

In Section \ref{sec_dirsum}, we embed linearly isometrically into $L^p$ spaces some normed vector spaces isometrically isomorphic to a direct sums of $L^p$ spaces thanks to the following basic result:

\begin{df}
 Let $I$ be a countable index set, $(B_i,\|.\|_{_{B_i}})_{i \in I}$ be a family of Banach spaces and $p \geq 1$.
 We call \emph{$\ell^p$-direct sum} of the family $(B_i)$ the space:
 $$ B=\rexp{p}{\bigoplus_{i \in I}} B_i:=\left \lbrace (x_i)_{i \in I} \in \prod_{i \in I}B_i \st \sum_{i \in I}\|x_i\|_{_{B_i}}^p  < +\infty \right \rbrace, $$
 and we denote, for $x=(x_i) \in B$,
 $$ \|x\|_p:=\left(\sum_{i \in I}\|x_i\|_{_{B_i}}^p\right)^{\frac{1}{p}}.$$
\end{df}

The space $B=\rexp{p}{\bigoplus_{i \in I}} B_i$ endowed with the norm $\|.\|_p$ is a Banach space, and moreover, we have:

\begin{prop}\label{lpisomisom}
 Let $I$ be a countable index set, $p \geq 1$ and $(L^p(X_i,\mu_i))_{i \in I}$ be a family of $L^p$-spaces.
 Then $\displaystyle \left(\rexp{p}{\bigoplus_{i \in I}} L^p(X_i,\mu_i),\|.\|_p\right)$ is isometrically isomorphic to a $L^p$-space.
\end{prop}

\section{Spaces with labelled partitions and actions on Banach Spaces}\label{sec_lab}
In this section we will introduce the structure of \emph{space with labelled partitions} and record for further use a few basic properties.

 \subsection{Spaces with labelled partitions}\label{subsec_lab}
 
 \subsubsection{Definitions}\label{subsubsec_deflab}
 %Définitions
Let $\K=\R$ or $\C$. \smallskip \\
Consider a set $X$ and a function $p:X\rightarrow \K$. There is a natural partition $P=P(p)$ of $X$ associated with $p$: \\
We have the following equivalence relation $\sim_p$ on $X$: for $x,y \in X$, 
\begin{center} 
 $x \sim_p y$ if, and only if, $p(x)=p(y)$.
\end{center}
We define the partition associated with $p$ by $\displaystyle P(p)= \{ \pi_p^{-1}(h) \st h \in X/ \sim_p \}$ where $\pi_p$ is the canonical projection from $X$ to $X/ \sim_p$.

\begin{df}\label{labpart}
Let $X$ be a set, and $\P=\{ p: X \rightarrow \K \}$ be a family of functions. 
\begin{itemize}
 \item We say that $p$ is a \emph{labelling function} on $X$ and the pair $(P,p)$ is called a \emph{labelled partition} of $X$.\\
 \item We say that $x,y \in X$ are \emph{separated} by $p \in \P$ if $p(x)\neq p(y)$ and we denote by $\P(x|y)$ the set of all labelling functions separating $x$ and $y$.
\end{itemize}

\end{df}

\begin{rmq}\label{labpart_rmq}
 The terminology ``$x$ and $y$ are separated by $p$'' comes from the fact that, if we denote by $P$ the partition of $X$ associated with $p$, $x$ and $y$ are separated by $p$ if, and only if, $x$ and $y$ belongs to two different sets of the partition $P$ i.e. $P$ separates $x$ and $y$.
\end{rmq}

Consider a set $\P$ of labelling functions on $X$, and the $\K$-vector space $\F(\P,\K)$ of all functions from $\P$ to $\K$. Then we have a natural map $c:X\times X \rightarrow \F(\P,\K)$ given by: for $x,y \in X$ and $p \in \P$,
$$c(x,y)(p)=p(x)-p(y).$$ \smallskip 
Notice that $p$ belongs to $\P(x|y)$ if, and only if, $c(x,y)(p)\neq 0$.  

\begin{df}\label{labpart_sepmap}
 Let $X$ be a set and $\P$ be a family of labelling functions. The map $c: X \times X \rightarrow \F(\P,\K)$ such that, for $x,y \in X$ and for $p \in \P$, $c(x,y)(p)=p(x)-p(y)$ is called the \emph{separation map} of $X$ relative to $P$. 
\end{df}

We now define the notion of space with labelled partitions:

\begin{df}[Space with labelled partitions]\label{labpart_space} $\;$ \\
 Let $X$ be a set, $\P$ be a family of labelling functions from $X$ to $\K$ and $(\F(\P),\|.\|)$ be a semi-normed space of $\K$-valued functions on $\P$ such that the quotient vector space $F(\P)$ of $\F(\P)$ by its subspace $\F(\P)_0=\{\k \in \F(\P) \st \|\k\|=0 \}$ is a Banach space. \medskip \\
 We say that $(X,\P,F(\P))$ is a space with labelled partitions if, for all $x,y \in X$: 
 $$c(x,y): \P \rightarrow \K \text{ belongs to }\F(\P).$$
\end{df}

\begin{df}\label{labpart_metric}
 If $(X,\P,F(\P))$ is a space with labelled partitions, we can endow $X$ with the following pseudo-metric: $d(x,y)=\|c(x,y)\|$ for $x,y \in X$. \\
 We call $d$ the \emph{labelled partitions pseudo-metric} on $X$.
\end{df}

\begin{rmq}\label{labpart_metric_rmq}
 If $(X,\P,F(\P))$ is a space with labelled partitions, then the separation map $c:X\times X \rightarrow F(\P)$ is continuous where $X\times X$ is endowed with the product topology induced by the topology of $(X,d)$.
\end{rmq}
 
 \subsubsection{Actions on spaces with labelled partitions}\label{subsubsec_defactlab}
 %Action sur un espace à partitions pondérées et définition continuité, propreté d'une action
 Here, we describe the maps that preserve the structure of space with labelled partitions.
 
  \begin{df}[homomorphism of spaces with labelled partitions]\label{labpart_homo}
  Let $(X,\P,F(\P))$, $(X',\P',F'(\P'))$ be spaces with labelled partitions and let $f:X\rightarrow X'$ be a map from $X$ to $X'$. \smallskip \\
  We say that $f$ is a \emph{homomorphism of spaces with labelled partitions} if: \\
  \begin{enumerate}
   \item for any $p' \in \P'$, $\Phi_f(p'):=p'\circ f$ belongs to $\P$,
   \item for all $\k \in F(\P)$, $\k\circ \Phi_f$ belongs to $F'(\P')$ and, 
   $$\|\k\circ \Phi_f\|_{_{F'(\P')}}=\|\k\|_{_{F(\P)}}.$$
  \end{enumerate} \smallskip 
  An \emph{automorphism} of the space with labelled partitions $(X,\P,F(\P))$ is a bijective map $f:X\rightarrow X$ such that $f$ and $f^{-1}$ are homomorphisms of spaces with labelled partitions from $(X,\P,F(\P))$ to $(X,\P,F(\P))$.
 \end{df}
 
 \begin{rmq}\label{labpart_homo_rmq} $\;$ \\
  - If $f$ is a homomorphism of spaces with labelled partitions, then $f$ is an isometry from $X$ to $X'$ endowed with their respective labelled partitions pseudo-metrics; indeed, for $x,y \in X$, $$d_X(x,y)=\|c(x,y)\|_{_{F(\P)}}=\|c(x,y)\circ \Phi_f\|_{_{F'(\P')}}=\|c'(f(x),f(y))\|_{_{F'(\P')}}=d_{X'}(f(x),f(y)),$$ since we have $c(x,y)\circ \Phi_f=c'(f(x),f(y))$. \smallskip \\ 
  - If $f$ is an automorphism of space with labelled partitions, the map $\Phi_f$ is a bijection: $(\Phi_f)^{-1}=\Phi_{f^{-1}}$. \\
 \end{rmq}
 
 \begin{prop}\label{comp_homo}
 Let $(X,\P,F(\P))$, $(X',\P',F'(\P'))$,$(X'',\P'',F''(\P''))$ be spaces with labelled partitions and $f:X\rightarrow X',\; f':X'\rightarrow X''$ be homomorphisms of spaces with labelled partitions. \\
 We denote $\Phi_f$ the map such that $\Phi_f(p'):=p'\circ f$, for $p' \in \P'$, and $\Phi_{f'}$ the map such that $\Phi_{f'}(p''):=p''\circ f'$, for $p'' \in \P''$. \smallskip \\
 Then $f' \circ f$ is a homomorphism of spaces with labelled partitions from $(X,\P,F(\P))$ to $(X'',\P'',F''(\P''))$ and we have, by denoting $\Phi_{f'\circ f}(p''):=p''\circ (f'\circ f)$:
 $$ \Phi_{f} \circ \Phi_{f'}=\Phi_{f'\circ f}.$$
 \end{prop}
 
 \begin{proof}
  For all $p'' \in \P''$, we have: 
  \begin{center}
   \begin{tabular}{rl}
    $\Phi_{f'\circ f}(p'')$&$=p''\circ (f'\circ f)$ \smallskip \\
    &$=(p''\circ f')\circ f$  \smallskip \\
    &$=\Phi_{f'}(p'') \circ f$ with $\Phi_{f'}(p'') \in \P'$ by Definition \ref{labpart_homo}  \smallskip \\
    &$=\Phi_f(\Phi_{f'}(p''))$ and hence, \medskip \\
    $\Phi_{f'\circ f}(p'')$&$=\Phi_f\circ \Phi_{f'}(p'') \in \P$ by Definition \ref{labpart_homo}.
   \end{tabular}
  \end{center} \smallskip
  It follows that $\Phi_{f} \circ \Phi_{f'}=\Phi_{f'\circ f}$. \\
  
  Now, let $\k \in F(\P)$. Since $\k \circ \Phi_{f}$ belongs to $F'(\P')$, 
  $$\k \circ \Phi_{f'\circ f}= (\k \circ \Phi_{f}) \circ \Phi_{f'} \in F''(\P''),$$
  and we clearly have, using the previous equality, 
  $$\|\k \circ \Phi_{f'\circ f}\|_{_{F''(\P'')}}= \|\k \circ \Phi_{f}\|_{_{F'(\P')}}=\|\k\|_{_{F(\P)}}.$$

 \end{proof}
 
 \begin{rmq}\label{comp_homo_rmq} Assume a group $G$ acts by automorphisms on $(X,\P,F(\P))$. For $g \in G$, we denote by $\tau(g): X \rightarrow X$, the map $x \mapsto \tau(g)x=gx$. Then, by Proposition  \ref{comp_homo}, we have: 
  $$\Phi_{\tau(g_2)}\circ \Phi_{\tau(g_1)} = \Phi_{\tau(g_1g_2)}.$$ 
 \end{rmq}

 \begin{df}\label{labpart_propertiesaction}
  Let $(X,\P,F(\P))$ be a space with labelled partitions and $G$ be a topological group acting by automorphisms on $(X,\P,F(\P))$. \medskip 
  \begin{itemize}
   \item We say that $G$ acts \emph{continuously} on $(X,\P,F(\P))$, if the $G$-action on $(X,d)$ is strongly continuous. \\
   \item We say that $G$ acts \emph{properly} on $(X,\P,F(\P))$, if the $G$-action on $(X,d)$ is metrically proper where $d$ is the labelled partitions pseudo-metric on $X$.
  \end{itemize}
 \end{df}

 \begin{rmq}\label{labpart_actionproper_rmq}
  Notice that if a topological Hausdorff group $G$ acts properly continuously by automorphisms on a space $(X,\P,F(\P))$ with labelled partitions, then it is locally compact and $\s$-compact: in fact, let $x_0 \in X$; for $r>0$, $V_r=\overline{\{ g \in G \st d(gx_0,x_0) \leq r \}}$ is a compact neighbourhood of the identity element $e$ in $G$ since the action on $(X,d)$ is strongly continuous and proper, and we have $G=\cup_{n \in\N^*}V_n$.
 \end{rmq}
 
 \begin{prop}\label{labpart_proper_prop}
  Let $G$ be a topological group. Assume $G$ acts continuously by automorphisms on $(X,\P,F(\P))$. \smallskip \\
  The $G$-action on $(X,\P,F(\P))$ is proper if, and only if, for every (resp. for some) $x_0 \in X$, $\|c(gx_0,x_0)\| \rightarrow \infty$ when $g\rightarrow \infty$.
 \end{prop}
 
 \begin{proof}
  It follows immediatly from the definition of a metrically proper action.
 \end{proof}
 
 \begin{lem}[pull back of space with labelled partitions]\label{pullbackpart}
  Let $(X,\P_X,F_X(\P_X))$ be a space with labelled partitions, $Y$ be a set and $f:Y\rightarrow X$ be a map. Then there exists a \emph{pull back} structure of space with labelled partitions $(Y,\P_Y,F_Y(\P_Y))$ turning $f$ into a homomorphism. \smallskip \\
  Moreover, if $G$ acts on $Y$ and $G$ acts continuously by automorphisms on $(X,\P_X,F_X(\P_X))$ such that $f$ is $G$-equivariant, then $G$ acts continuously by automorphisms on $(Y,\P_Y,F_Y(\P_Y))$.
 \end{lem}
 
 \begin{proof}
  We consider the family of labelling functions on $Y$ :
  $$ \P_Y=\{ p\circ f \st p \in \P_X\}, $$
  and let $c_Y$ be the separation map on $Y$ associated with $\P_Y$. \\
  Let $T: \Vect\big(c_Y(y,y') \st y,y' \in Y\big) \rightarrow F_X(\P_X)$ be the linear map such that $T(c_Y(y,y'))=c_X(f(y),f(y'))$.
  The map $T$ is well defined and is injective since, for every $p \in \P_X$,
  $$c_X(f(y),f(y'))(p)=p\circ f(y)-p\circ f(y')=c_Y(y,y')(p\circ f). $$
  On $\Vect\big(c_Y(y,y') \st y,y' \in Y\big)$, we consider the following norm : \smallskip \\
  for $\k \in \Vect\big(c_Y(y,y') \st y,y' \in Y\big)$, we set,
  $$ \|\k\|_{_{\P_Y}}=\|T(\k)\|_{_{F_X(\P_X)}}. $$
  And we set $F_Y(\P_Y)=\overline{\Vect\big(c_Y(y,y') \st y,y' \in Y\big)}^{\|\cdotp\|_{_{\P_Y}}}$. Hence, by construction, $(Y,\P_Y,F_Y(\P_Y))$ is a space with labelled partitions and $f$ is clearly an homorphism from $(Y,\P_Y,F_Y(\P_Y))$ to $(X,\P_X,F(\P_X))$ since, for all $y,y' \in Y$,
  $$ c_Y(y,y')\circ \Phi_f=c_X(f(y),f(y')), $$
  where $\Phi_f(p)=p\circ f$ for $p \in \P_X$. \smallskip \\

  Assume that $G$ acts on $Y$ via $\tau_Y$ and $G$ acts continuously by automorphisms on $(X,\P,F(\P))$ via $\tau_X$, and $f$ is $G$-equivariant. We denote, for $p \in \P_X$ and $g \in G$: \\
  - $\Phi_{\tau_X(g)}(p):=p \circ \tau_X(g)$ and, \\
  - $\Phi_{\tau_Y(g)}(p\circ f):=(p \circ f) \circ \tau_Y(g)$. \smallskip \\
  Since $f$ is $G$-equivariant and $\P_X$ is stable by $\tau_X$, we have, for all $p \in \P_X$ and all $g \in G$ :
  $$ (p \circ f) \circ \tau_Y(g)=(p \circ \tau_X(g))\circ f \in \P_Y .$$
  Now, for every $\k \in F_Y(\P_Y)$ and every $g \in G$, we have :
  $$ \|\k\circ \Phi_{\tau_Y(g)}\|_{_{\P_Y}}=\|T(\k)\circ \Phi_{\tau_X(g)} \|_{_{F_X(\P_X)}}=\|T(\k)\|_{_{F_X(\P_X)}}=\|\k\|_{_{\P_Y}}. $$
  It follows that $G$ acts by automorphisms on $(Y,\P_Y,F_Y(\P_Y))$. \\
  Moreover, we have, for every $y \in Y$ and every $g \in G$, $d_Y(\tau_Y(g)y,y)=d_X(\tau_X(g)f(y),f(y))$, where $d_X$ and $d_Y$ are the labelled partitions pseudo-metric on respectively $X$ and $Y$. Hence, for $y \in Y$, $y \rightarrow \tau_Y(g)y$ is continuous from $G$ to $(Y,d_Y)$. 
 \end{proof}

 \begin{df}\label{pullbackk}
  Let $(X,\P_X,F_X(\P_X))$ be a space with labelled partitions, $Y$ be a set and $f:Y\rightarrow X$ be a map. The structure of space with labelled partitions $(Y,\P_Y,F_Y(\P_Y))$ given by Lemma \ref{pullbackpart} is called the \emph{pull back by $f$ of the space with labelled partitions $(X,\P_X,F_X(\P_X))$}.
 \end{df}

 \subsection{Examples}
 
 \subsubsection{Spaces with measured walls}
 
 %Exemples: espaces à murs mesurés, Yu, espace naif.
 Our first example of spaces with labelled partitions is given by spaces with measured walls. Here we cite the definition of the structure of space with measured walls from \cite{wreath}. \smallskip \\
 Let $X$ be a set. We endow $2^X$ with the product topology and we consider, for $x \in X$, the clopen subset of $2^X$, $\A_x:=\{ A \subset X \st x \in A\}$.
 \begin{df}\label{labpart_meswalls}
A \emph{measured walls structure} is a pair $(X, \mu)$ where $X$ is a set and
$\mu$ is a Borel measure on $2^X$ such that for all $x, y \in X$:
$$ d_{\mu}(x,y):=\mu(\A_x\ds \A_y)<+\infty $$
\end{df}

\begin{prop}\label{walls_labpart}
 Let $(X,\mu)$ be a measured space with walls. Then, for every real number $q \geq 1$, $(X,\P,L^q(\P,\mu))$ is a space with labelled partitions where $\P=\{ \cara_h \st h \in 2^X \}$. \\
 Morever, we have, for $x,y \in X$, 
 $$ \|c(x,y)\|_q^q=d_{\mu}(x,y) .$$
\end{prop}

\begin{proof}
 We denote $\P=\{ \cara_h \st h \in 2^X \}$. Then $\P$ is a family of labelling functions on $X$ and we denote by $c$ the separation map of $X$ associated with $\P$. \\
 Let $x,y \in X$. For $h \in 2^X$, we have: 
 $$c(x,y)(\cara_h)=\cara_h(x)-\cara_h(y)=\cara_{\A_x}(h)-\cara_{\A_y}(h).$$
 The function $f:2^X \rightarrow \P$ such that, for $h \in 2^X$, $f(h)=\cara_h$ is a bijection, and we endow $\P$ with the direct image topology induced by $f$. Then, $\mu^*:\P\rightarrow \R$ such that, for any Borel subset $A$ of $\P$, $\mu^*(A)=\mu(f^{-1}(A))$ is a Borel measure on $\P$. \smallskip \\
 We have $\|c(x,y)\|_q^q=\int_{\P} |c(x,y)(p)|^q d\mu^*(p)=\int_{2^X} |\cara_{\A_x}(h)-\cara_{\A_y}(h)|^q d\mu(h)=\mu(\A_x\ds \A_y)$, and then:
 $$\|c(x,y)\|_q^q=d_{\mu}(x,y) < +\infty .$$
 It follows that, for all $x,y \in X$, $c(x,y)$ belongs to $L^q(\P,\mu)$ and hence, $(X,\P,L^q(\P,\mu))$ is a space with labelled partitions.
\end{proof}

 \begin{center}
 \includegraphics[width=11cm]{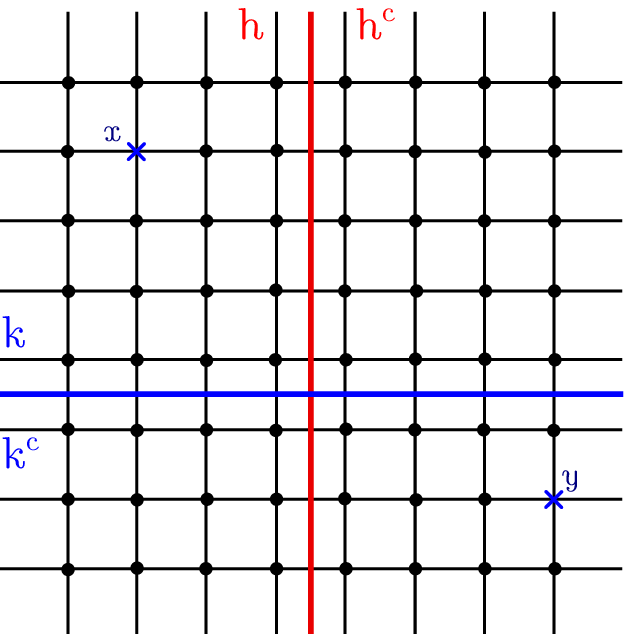} \\
  Examples of walls in $\Z^2$.
 \end{center}
 
 \subsubsection{Gromov hyperbolic groups}

% The following example is given by a Theorem of Yu in \cite{yu}, saying that every finitely generated Gromov hyperbolic group is a-$FL^p$-menable (with $p\geq 2$). The definition of labelled partitions was motivated by this result... \{color{blue} A dire dans l'intro plutot...} \\
% The following example is coming from 

The following Lemma is a reformulation of a result of Yu (see \cite{yuhyp}, Corollary 3.2) based on a construction of Mineyev in \cite{minhyp}. \\
For a triple $x,y,z$ in a metric space $(X,d)$, we denote by $(x|y)_z=\frac{1}{2}(d(x,z)+d(y,z)-d(x,y))$. 
\begin{lem}[Mineyev, Yu]\label{hyp_lem}
Let $\G$ be a finitely generated $\d$-hyperbolic group. Then there exists a $\G$-equivariant function $h:\G \times \G \rightarrow \F_c(\G)$ where $\F_c(\G)=\{ f:\G\rightarrow \R \text{ with finite support}\st \|f\|_1=1\}$ such that:
\begin{enumerate}
 \item[1.] for all $a,x \in \G$, $\su h(x,a) \subset B(a,10\d)$,
 \item[2.] there exists constants $C \geq 0$ and $\eps > 0$ such that, for all $x,x',a \in \G$,
 $$\| h(x,a)-h(x',a)\|_1\leq C e^{-\eps (x|x')_a}, $$
 \item[3.] there exists a constant $K \geq 0$ such that, for all $x,x' \in \G$ with $d(x,x')$ large enough, 
 $$\#\{ a \in \G \st \su h(x,a) \cap \su h(x',a)=\emptyset \}\geq d(x,x')-K. $$
\end{enumerate}

\end{lem}

 \begin{center}
 \includegraphics[width=11cm]{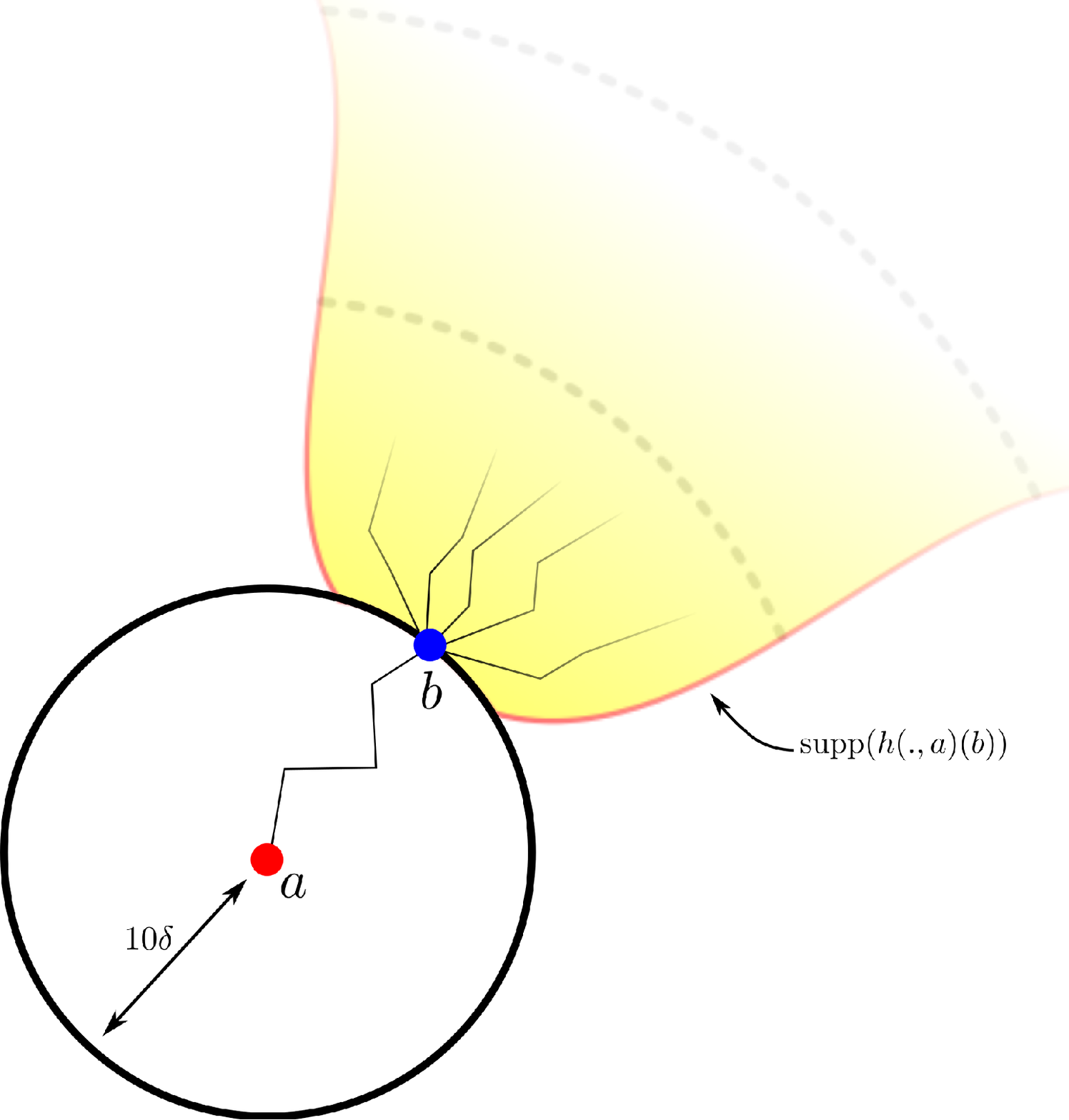} \\
  Support of the labelling function associated with $(a,b)$ with $d(a,b)=10\d$.
 \end{center}
 
 This Lemma gives us a way to build a structure of labelled partitions on Gromov hyperbolic groups:

\begin{prop}[Labelled partitions on a $\d$-hyperbolic group]\label{hyp_labpart}
 Let $\G$ be a finitely generated $\d$-hyperbolic group and we denote $\P=\{(a,b) \in \G \times \G \st d(a,b)\leq 10\d \}$.
 There exists $q_0\geq 1$ such that, for all $q > q_0$, $(\G,\P,\ell^q(\P))$ is a space with labelled partitions.
\end{prop}

\begin{rmq}
 Notice that, stated this way, $\P$ is not a set of labelling functions on $\G$. Implicitely, we do the following identification :
 $$\{(a,b) \in \G \times \G \st d(a,b)\leq 10\d \} \sim \{x \mapsto h(a,x)(b) \st (a,b) \in \G^2 \text{ with }d(a,b)\leq 10\d \}.$$
 In fact, $x \mapsto h(a,x)(b)$ is uniquely determined by the pair $(a,b)$. 
\end{rmq}

\begin{proof}[Proof of Proposition \ref{hyp_labpart}]
 We fix a finite generating set of $\G$ and we denote $d$ the word metric associated with it (and such that $\G$ is Gromov hyperbolic of constant $\d$ with respect to $d$). As $\G$ is uniformly locally finite, there exists a constant $k>0$ such that, for all $r>0$ and $x \in \G$, $\#B(x,r)\leq k^r$. \\
 
 Let $\eps$ be as in 2. Lemma \ref{hyp_lem} and set $q_0 = \frac{\text{ln}(k)}{\eps}$. Let $q > q_0$. Then for all $q > q_0$, 
 $$ \sum_{n \in \N}k^ne^{-nq\eps} < +\infty. $$
 Let $h$ be the function given by Lemma \ref{hyp_lem} and notice that, for $x,x',a \in \G$, since $\# \supp(h(x,a)) \leq k$, 
 $$ \|h(x,a)-h(x',a)\|_q \leq 2k^{\frac{1}{q}}\|h(x,a)-h(x',a)\|_1. \quad \quad (*)$$
 As said in the previous remark, we can see $\P$ as a set of labelling functions on $\G$ using the function $h$ : we set, for $(a,b) \in \P$ and $x \in \G$, 
 $$(a,b)(x):=h(x,a)(b).$$
 We denote by $c$ the separation map associated with $\P$. We have, for $x,x' \in \G$, 
 \begin{center}
  \begin{tabular}{rcl}
   $\displaystyle \|c(x,x')\|_{\ell^q(\P)}^q$&$=$&$\displaystyle \sum_{(a,b) \in \P}|h(x,a)(b)-h(x,a)(b)|^q,$ \smallskip \\
   &$=$&$\displaystyle \sum_{a \in \G}\|h(x,a)-h(x,a)\|_q^q$ by 1. Lemma \ref{hyp_lem}, \smallskip \\
   &$\leq$&$\displaystyle \sum_{a \in \G}2^q k\|h(x,a)-h(x,a)\|_1^q$ by $(*)$, \smallskip \\
   &$\leq$&$\displaystyle  (2C)^qk\sum_{a \in \G} e^{-q\eps (x|x')_a}$ by 2. Lemma \ref{hyp_lem}, \smallskip \\
   &$\leq$&$\displaystyle  (2C)^qk\sum_{a \in \G} e^{-q\eps (d(x,a)-d(x,x'))}$, \smallskip \\
   &$\leq$&$\displaystyle (2C)^qk\sum_{n \in \N} k^ne^{-q\eps (n-d(x,x'))}$, and hence, since $q>q_0$:\smallskip \\
   $\displaystyle \|c(x,x')\|_{\ell^p(\P)}^p$&$\leq$&$\displaystyle (2C)^q e^{q\eps d(x,x')}< +\infty$  \smallskip \\
  \end{tabular}
 \end{center}
 Thus $c(x,x')$ belongs to $\ell^p(\P)$ for all $x,x' \in \G$.
 It follows that $(\G,\P,\ell^p(\P))$ is a space with labelled partitions.
\end{proof}

\begin{prop}\label{hyp_actionprop}
  Let $\G$ be a finitely generated $\d$-hyperbolic group. Let $q_0 \geq 1$ as in Proposition \ref{hyp_labpart} and for $q >q_0$, let $(\G,\P,\ell^q(\P))$ be the space with labelled partitions given by Proposition \ref{hyp_labpart}.
  Then the action of $\G$ by left-translation on itself induces a proper action of $\G$ by automorphisms on $(\G,\P,\ell^q(\P))$.
 \end{prop}
 
 \begin{proof}
  We keep the notations used in the proof of Proposition \ref{hyp_labpart}. We first show that $\G$ acts by automorphisms on $(\G,\P,\ell^q(\P))$.
  Let $\g,x \in \G$ and $(a,b) \in \P$. Since $h$ is $\G$-equivariant, we have:
  $$\Phi_{\g}((a,b))(x)=(a,b)(\g x)=h(\g x,a)(b)=h(x,\g^{-1}a)(\g^{-1}b)=(\g^{-1}a,\g^{-1}b)(x),$$
  And hence, $$\Phi_{\g}((a,b))=(\g^{-1}a,\g^{-1}b) \in \P.$$
  Moreover, for $\k \in \ell^q(\P)$, we have:
  \begin{center}
   \begin{tabular}{rl}
    $\|\k \circ \Phi_{\g}\|_{_{\ell^q(\P)}}^q $&$\displaystyle=\sum_{(a,b)\in \P}|\k(\g^{-1}a,\g^{-1}b)|^q$, \smallskip \\
    &$\displaystyle=\sum_{(\g a,\g b)\in \P}|\k(a,b)|^q$, \smallskip \\
    &$\displaystyle=\sum_{(a,b)\in \P}|\k(a,b)|^q$, \smallskip \\
    $\|\k \circ \Phi_{\g}\|_{_{\ell^q(\P)}}^q $&$=\|\k\|_{_{\ell^q(\P)}}^q$.
   \end{tabular}
  \end{center}
  It follows that $\G$ acts by automorphisms on $(\G,\P,\ell^q(\P))$. \medskip \\
  Now, consider the identity element $e$ of $\G$ and let $\g \in \G$. \\
  We denote $A=\{ a \in \G \st \su h(\g,a) \cap \su h(e,a)=\emptyset \}$. Notice that for every $x,a \in \G$, $\|h(x,a)\|_q\geq \frac{1}{k}$. We have, by 3. Lemma \ref{hyp_lem}, when $d(\g,e)$ is large enough:
  \begin{center}
   \begin{tabular}{rcl}
    $\|c(\g,e)\|_{_{\ell^q(\P)}}^q $&$=$&$\displaystyle\sum_{a\in \G}\|h(\g,a)-h(e,a)\|_q^q$, \smallskip \\
    &$\geq$&$\displaystyle \sum_{a\in A}\|h(\g,a)-h(e,a)\|_q^q\geq\sum_{a\in A}(\frac{2}{k})^q$, since $\|h(x,a)\|_q\geq \frac{1}{k}$ \medskip \\
    $\|c(\g,e)\|_{_{\ell^q(\P)}}^q $&$\geq$&$\displaystyle (\frac{2}{k})^q(d(\g,e)-K)$.
   \end{tabular}
  \end{center}
And hence, when $\g \rightarrow \infty$ in $\G$, we have: $\|c(\g,e)\|_{_{\ell^q(\P)}}^q\geq (\frac{2}{k})^q(d(\g,e)-K)\rightarrow +\infty$. 
 \end{proof}
 
 \subsubsection{Labelled partitions on metric spaces}
 
 It turns out that any pseudo-metric spaces $(X,d)$ can be realized as a space with labelled partitions $(X,\P,F(\P))$ with $F(\P)\simeq \ell^{\infty}(X)$ and such that the pseudo-metric of labelled partitions is exactly $d$ :
 \begin{prop}
  Let $(X,d)$ be a pseudo-metric space and consider the family of labelling functions on $X$ : 
  $$\P=\{ p_z:x \mapsto d(x,z) \st z \in X\}.$$
  Then $(X,\P,\ell^{\infty}(\P))$ is a space with labelled partitions. \smallskip \\
  Moreover, for all $x,y \in X$,
  $$ d_{\P}(x,y)=d(x,y), $$
  where $d_{\P}$ is the pseudo-metric of labelled partitions on $X$.
 \end{prop}
 
 \begin{proof}
  Let $c$ be the separation map associated with $\P=\{ p_z:x \mapsto d(x,z) \st z \in X\}$. For $x,y \in X$ and $p_z \in \P$,  we have :
  $$ c(x,y)(p_z)=p_z(x)-p_z(y)=d(x,z)-d(y,z)\leq d(x,y),$$
  and, in particular, $c(x,y)(p_y)=d(x,y)$, then, 
  $$ \|c(x,y)\|_{\infty}=\sup_{p_z \in \P}|c(x,y)(p_z)|=d(x,y). $$
  Hence, $(X,\P,\ell^{\infty}(\P))$ is a space with labelled partitions and $d_{\P}(x,y)=\|c(x,y)\|_{\infty}=d(x,y)$.  
 \end{proof}
 This result motivates the study of structures of spaces labelled partitions on a pseudo-metric space $X$ : can we find other Banach spaces than $\ell^{\infty}(X)$ which gives a realization of the pseudo-metric on $X$ as a pseudo-metric of labelled partitions ? \\

A first element of answer is given by the case of the discrete metric on a set. On every set, we can define a structure of labelled partitions which gives the discrete metric on this set:

\begin{prop}\label{naivelabpart_prop}
  Let $X$ be a set and $\P=\{ \di_x \st x \in X\}$ be the family of labelling functions where, for $x \in X$, $\di_x=2^{-\frac{1}{q}}\d_x$. \\
  Then, for every $q \geq 1$, $(X,\P,\ell^q(\P))$ is a space with labelled partitions.
\end{prop}

\begin{proof}
We have, for $x,y,z \in X$ with $x\neq y$: \\
 
 \begin{equation*}
   c(x,y)(\di_z)=\di_z(x)-\di_z(y)=\begin{cases} 0 \text{ if } z\notin \{x,y\}  \\
                            \pm 2^{-\frac{1}{q}} \text{ otherwise.}
                           \end{cases}
 \end{equation*}
and then,

$$ \|c(x,y)\|^q_q=\sum_{z \in X}|c(x,y)(\di_z)|^q=|c(x,y)(\di_x)|^q+|c(x,y)(\di_y)|^q=1. $$
\end{proof}

Notice that the labelled partitions pseudo-metric $d$ on $X$ in this case is precisely the discrete metric on $X$ i.e. $d(x,y)=1$ for all $x,y \in X$, $x \neq y$.

\begin{df}[Naive $\ell^q$ space with labelled partitions]\label{naivelabpart}
 Let $X$ be a set and $\P=\{ \di_x \st x \in X\}$.\smallskip \\ For $q \geq 1$, $(X,\P,\ell^q(\P))$ is called the \emph{naive $\ell^q$ space with labelled partitions of $X$}.
\end{df}

 \begin{rmq}\label{labpart_homo_ex}
  Let $X$ be a set, $q \geq 1$ and $G$ a group acting on $X$. Then $G$ acts by automorphisms on the naive $\ell^q$ space with labelled partitions of $X$. \smallskip \\
  In fact, if, for $g \in G$, we denote $\tau(g): x \mapsto gx$, we have, for $z \in X$, 
  $$\di_z \circ \tau(g) = \di_{g^{-1}z} \in \P,$$
  and, for all $\k \in \ell^q(\P)$, 
  $$\|\k\circ \Phi_{\tau(g)}\|^q_q=\sum_{x\in X}|\k(\di_{gx})|^q=\sum_{g^{-1}x\in X}|\k(\di_{x})|^q=\sum_{x\in X}|\k(\di_{x})|^q=\|\k\|_q^q.$$
 \end{rmq}
 
 \subsubsection{Labelled partitions on Banach spaces}
Every Banach space has a natural structure of space with labelled partitions and the metric of labelled partitions of this structure is exactly the metric induced by the norm. \\
Let $f$ be a $\K$-valued function on a set $B$ and $k \in \K$. We denote $f+k:=\{x \mapsto f(x)+k\}$.

\begin{df}
 Let $B$ be a Banach space and $B'$ be its topological dual. The set : 
 $$ \P=\{f+k \st f \in B', \; k \in \K\}$$ 
 is called the \emph{natural family of labelling functions} on $B$. \\
 Let $c$ be the separation map on $B$ associated with $\P$. We denote :
 $$ \d(\P)=\{c(x,x')\st x,x' \in B \}. $$
\end{df}

\begin{rmq}
This definition and the fact that the natural family of labelling functions contains the constant functions are motivated by the following : as we shall see in Lemma \ref{actcanbanach}, a $G$-action on a Banach space $B$ by affine isometries induces an action of $G$ on the natural family of labelling functions on $B$.
\end{rmq}

\begin{prop}\label{banachspacelabpart}
Let $(B,\|\cdotp\|)$ be a Banach space and $\P$ be its natural family of labelling functions. Then $\d(\P)$ is isomorphic to $B$ and $(B,\P,\d(\P))$ is a space with labelled partitions where $\d(B)$ is viewed as an isometric copy of $B$.
Moreover, we have, for $x,x' \in B$ :
$$ d(x,x')=\|x-x'\|, $$
where $d$ is the pseudo-metric of labelled partitions on $(B,\P,\d(\P))$.
\end{prop}

\begin{proof}
Let $\P:=\{f+k \st f \in B', \; k \in \K\}$ and let $c$ be the separation map on $B$ associated with $\P$. Notice that for all $x,x' \in B$, $c(x-x',0)=c(x,0)-c(x',0)=c(x,x')$. Then the map $T: B \rightarrow \d(B)$ such that $x \mapsto c(x,0)$ is clearly a surjective linear operator. Now, we have $c(x,0)=0\Leftrightarrow \forall f \in B', \; f(x)=0$, and hence, by Hahn-Banach Theorem, $T$ is injective. It follows that $T$ is an isomorphism. \\
The quantity $\|c(x,x')\|_{_{\d(\P)}}:=\|x-x'\|$ defines a norm on $\d(\P)$ and hence, $(\d(\P),\|.\|_{_{\d(\P)}})$ is a Banach space as $T$ is an isometric isomorphism. It follows immediately that $(B,\P,\d(\P))$ is a space with labelled partitions.
\end{proof}

\begin{df}
 Let $B$ be a Banach space. The space with labelled partitions $(B,\P,\d(\P))$ where $\P=\{f+k \st f \in B', \; k \in \K\}$ and $\d(\P)\simeq B$ is called the \emph{natural structure of labelled partitions on $B$}.
\end{df}

\begin{lem}\label{actcanbanach}
 Let $G$ be a topological group. Then a continuous isometric affine action of $G$ on a Banach space $B$ induces a continuous action of $G$ by automorphisms on the natural space with labelled partitions $(B,\P, \d(\P))$ on $B$.
\end{lem}

\begin{proof}
 Let $\a$ be a continuous isometric affine action of $G$ on a Banach space $B$ with linear part $\pi$ and translation part $b$. Let $(B,\P,\d(\P))$ be the natural space with labelled partitions on $B$. \\
 Notice that for all $f \in B'$, $f \circ \pi(g) \in B'$ since $\pi$ is an isometric representation. Hence, for all $g \in G$ and $p=f+k \in \P$ :
 $$p \circ \a(g)=f\circ \a(g)+k=f \circ \pi(g) + (k +f(b(g))) \in \P.$$
 We denote, for $g \in G$ and $p \in \P$, $\Phi_g(p)=p \circ \a(g)$. We have, for $g \in G$ and $c(x,x') \in \d(\P)$,
 \begin{center}
  \begin{tabular}{rcl}
   $ \|c(x,x')\circ \Phi_g\|_{_{\d(\P)}}$&$=$&$ \|c(\a(g)x,\a(g)x')\|_{_{\d(\P)}},$ \\
   &$=$&$ \|\a(g)x-\a(g)x'\|,$ \\
   &$=$&$ \|\pi(g)(x-x')\|,$ \\
   &$=$&$ \|x-x'\|,$ \\
   $ \|c(x,x')\circ \Phi_g\|_{_{\d(\P)}}$&$=$&$ \|c(x,x')\|_{_{\d(\P)}}.$
  \end{tabular}
 \end{center}
It follows that $G$ acts by automorphisms on $(B,\P, \d(\P))$ and this action is clearly continuous since $d(x,x')=\|x-x'\|$ where $d$ is the pseudo-metric of labelled partitions.
\end{proof}

In the particular case of a real Banach space $B$, we can consider another family of labelling functions on $B$ which is composed of functions valued in $\{0,1\}$; hence, it can be tought as characteristic functions of half spaces of the real Banach space $B$ :

\begin{dfpr}
 Let $B$ be a \emph{real} Banach space, $B'$ be its topological dual and $SB'$ its unit sphere (for the operator norm). For $f \in SB'$ and $k \in \R$, we define the function $p_{f,k}: B \rightarrow \{0,1\}$ by, for $x \in B$ :
  \begin{equation*}
   p_{f,k}(x)=\begin{cases} 1 \text{ if } f(x)-k >0  \\
                            0 \text{ otherwise.}
                           \end{cases}
 \end{equation*}
 We set $\P=\{p_{f,k} \st f \in SB',\; k \in \R \}$ and $\F(\P)=\Vect(c(x,y) \st x,y \in B)$ where $c$ is the separation map associated with $\P$.
 Then, for $\k \in \F(\P)$, the quantity $$\|\k\|:= \sup_{f \in SB'}\left|\int_{\R}\k(p_{f,k})dk\right|$$ is a semi-norm on $\F(\P)$ and $F(\P)=\F(\P)/\{\k \st \|\k\|=0\}$ is a Banach space isometrically isomorphic to $B$. Moreover, $(B,\P,F(\P))$ is a space with labelled partitions and we have, for all $x,y \in B$ :
 $$ d_{\P}(x,y)=\|c(x,y)\|=\|x-y\|_B. $$
\end{dfpr}

\begin{proof}
 First, notice that for $x,y \in B$ and $p_{f,k} \in \P$, we have :
 \begin{equation*}
   c(x,y)(p_{f,k})=p_{f,k}(x)-p_{f,k}(y)=\begin{cases} \pm 1 \text{ if } f(x)>k\geq f(y) \text{ or } f(y)>k\geq f(x) \\
                            0 \text{ otherwise.}
                           \end{cases}
 \end{equation*}
 Hence, for $f \in SB'$,
 \begin{equation*}
   \int_{\R}c(x,y)(p_{f,k})dk=\begin{cases} f(x)-f(y) \text{ if } f(x)\geq f(y) \\
                            f(y)-f(x) \text{ if } f(y)> f(x)
                           \end{cases}
 \end{equation*}
 It follows that $\|c(x,y)\|=\sup_{f \in SB'}|f(x)-f(y)|=\|x-y\|_B <+\infty$ $(*)$.\smallskip \\
 Now, for $\k=\sum \l_i c(x_i,y_i) \in \F(\P)$, $\|k\| \leq \sum |\l_i|\;\| c(x_i,y_i)\| < +\infty $, and then, $\|.\|$ is a semi-norm on $\F(\P)$. \\
 Let us now consider the quotient $F(\P)=\F(\P)/\sim$ where $\k\sim \k'$ if, and only if $\|\k-\k'\|=0$. For $\l,\mu \in \R$ and $x,y \in B$, we have $c(\l x +\mu y, 0) \sim \l c(x,0) + \mu c(y,0)$ and $c(x-y,0)\sim c(x,y)$. Thus, $T: B \rightarrow F(\P)$ such that $T(x)=c(x,0)$ is an isomorphism and by $(*)$ it is isometric. Hence, $F(\P)$ is a Banach space isometrically isomorphic to $B$ and $(B,\P,F(\P))$ is a space with labelled partitions. 
\end{proof}

 \subsection{Link with isometric affine actions on Banach spaces}\label{subsec_action}
 
  In this section, we aim to prove the two statements of Theorem \ref{labpart_affact} which gives an analog of the equivalence between proper actions on spaces with measured walls and Haagerup property in terms of proper actions on spaces with labelled partitions and isometric affine actions on Banach spaces; and more particularly in the case of $L^p$ spaces, using Hardin's result about extension of isometries on closed subspaces of $L^p$ spaces. \bigskip \\
  
  \begin{noth}[\ref{labpart_affact}]
   Let $G$ be topological group.
   \begin{enumerate}
    \item[1.] If $G$ acts (resp. acts properly) continuously by affine isometries on a Banach space $B$ then there exists a structure $(G,\P,F(\P))$ of space with labelled partitions on $G$ such that $G$ acts (resp. acts properly) continuously by automorphisms on $(G,\P,F(\P))$ via its left-action on itself. Moreover, there exists a linear isometric embedding $F(\P) \hookrightarrow B$.
    \item[2.] If $G$ acts (resp. acts properly) continuously by automorphisms on a space with labelled partitions $(X,\P,F(\P))$ then there exists a (resp. proper) continuous isometric affine action of $G$ on a Banach space $B$. Moreover, $B$ is a closed subspace of $F(\P)$.
   \end{enumerate}
  \end{noth}
  
   \begin{nocor}[\ref{labpart_aflp}]
  Let $p\geq 1$ with $p \notin 2\Z\smallsetminus \{2\}$ and $G$ be a topological group. $G$ has property $PL^p$ if, and only if, $G$ acts properly continuously by automorphisms on a space with labelled partitions $(X,\P,F(\P))$ where $F(\P)$ is isometrically isomorphic to a closed subspace of an $L^p$ space. 
 \end{nocor}
 
 \begin{proof}[Proof of Corollary \ref{labpart_aflp}]
   The direct implication follows immediately from 1) Theorem \ref{labpart_affact}. \smallskip \\
   Now, assume $G$ acts properly continuously by automorphisms on a space $(X,\P,F(\P))$ and $T:F(\P)\hookrightarrow L^p(X,\mu)$ is a linear isometric embedding. \\
   By 2) Theorem \ref{labpart_affact}, there is a proper continuous isometric affine action $\a$ of $G$ on a closed subspace $B$ of $F(\P)$ with $\a(g)=\pi(g)+b(g)$. Thus, as $T$ is a linear isometry, $T(B)$ is a closed subspace of $L^p(X,\mu)$ and $\a'$ such that $\a'(g)=T\circ\pi(g)\circ T^{-1}+T(b(g))$ is a continuous isometric affine action of $G$ on $T(B)$. Then, by Corollary \ref{hardin_lp_cor}, $G$ has property $PL^p$. 
 \end{proof}
 
  \subsubsection{Labelled partitions associated with an isometric affine action}\label{subsubsec_actionlab}
 In this part, we introduce the space with labelled partitions associated with a continuous isometric affine action of a topological group $G$ and we give a proof of 1) Theorem \ref{labpart_affact} by defining an action of $G$ by automorphisms on this structure. \medskip \\ 

Given a continuous isometric affine action on a Banach space, we consider the pullback of the natural structure of space with labelled partitions of the Banach space on the group itself :

\begin{df}\label{labpart_alpha_def}
 Let $G$ be a topological group and $\a$ be a continuous isometric affine action of $G$ on a Banach space $(B,\|.\|)$ with translation part $b:G \rightarrow B$. Consider the pullback $(G,\P_{\a},F_{\a}(\P_{\a}))$ by $b$ of the natural space with labelled partitions $(B,\P,\d(\P))$ on $B$, where $\P=B'$ and $\d(\P)\simeq B$. \\
 The triple $(G,\P_{\a},F_{\a}(\P_{\a}))$ is called \emph{the space with labelled partitions associated with $\a$}. \smallskip \\
 More precisely, we have : \\
   $\P_{\a} = \{f\circ b+k \st f \in B', \; k \in \K \}$; \smallskip \\
   $F_{\a}(\P_{\a}) \simeq \overline{\Vect(b(G))}^{\|.\|}$; \smallskip \\
\end{df}

\begin{rmq}\label{alpha_rmq} - The linear map $T:F_{\a}(\P_{\a}) \hookrightarrow B$ such that $T: c_{\a}(g,h) \mapsto b(g)-b(h)$ is an isometric embedding, where $c_{\a}$ is the separation map on $G$ associated with $\P_{\a}$.\smallskip \\
 - If the continuous isometric affine action $\a$ is linear i.e. $b(G)=\{ 0 \}$, then the space $(G,\P_{\a},F_{\a}(\P_{\a}))$ with labelled partitions associated with $\a$ is degenerated in the sense that the quotient metric space associated with $(G,d)$ contains a single point, $\P_{\a}$ contains only the zero function from $G$ to $\K$ and $F_{\a}(\P_{\a})=\{0\}$.
\end{rmq}

\begin{prop}\label{labpart_actiontolabpart}
   Let $G$ be a topological group and $(G,\P,F(\P))$ be the space with labelled partitions associated with a continuous isometric affine action of $G$ on a Banach space $B$.  \\
  Then the action of $G$ on itself by left-translation induces a continuous action of $G$ by automorphisms on $(G,\P,F(\P))$.
 \end{prop}
 
 \begin{proof} Let $\a$ be a continuous isometric affine action of $G$ on a Banach space $B$ with translation part $b: G \rightarrow B$. By Lemma \ref{actcanbanach}, $G$ acts continuously on the natural space with labelled partitions on $(B,\P,\d(\P))$ on $B$. Moreover, the map $b$ is $G$-equivariant since we have, for $g,h \in G$, $b(gh)=\a(g)b(h)$. By Lemma \ref{pullbackpart}, it follows that the $G$-action on itself by left-translation induces a continuous action by automorphisms on $(G,\P,F(\P))$.
 \end{proof}

\begin{proof}[Proof of 1) Theorem \ref{labpart_affact}]
   Assume $\a$ is continuous isometric affine action of $G$ on a Banach space $(B,\|.\|)$ with translation part $b$ and let $G$. \\
   By Proposition \ref{labpart_actiontolabpart}, the $G$-action by left-translation on itself induces a continuous action by automorphisms on the space with labelled partitions associated with $\a$, $(G,\P_{\a},F_{\a}(\P_{\a}))$. \\
   Moreover, assume $\a$ is proper. Then, by Remark \ref{alpha_rmq}, we have :
   $$ d_{\a}(g,e)=\|b(g)\| \underset{g \rightarrow \infty}{\longrightarrow} +\infty,$$
   and hence, the $G$-action by automorphisms on $(G,\P_{\a},F_{\a}(\P_{\a}))$ is proper. 
 \end{proof}
 
  \subsubsection{From actions on a space with labelled partitions to isometric affine actions}\label{subsubsec_labaction}
  We prove here statement 2) of Theorem \ref{labpart_affact} by giving a (non-canonical) way to build a proper continuous isometric affine action on a Banach space given a proper continuous action by automorphisms on space with labelled partitions.
  
  \begin{lem}\label{labpart_lptoaction_lem}
  Let $G$ be a topological group, $(X,\P,F(\P))$ be a space with labelled partitions and we denote $E=\Vect(c(x,y) \st x,y \in X)$ where $c$ is the separation map associated with $\P$.\\
  If $G$ acts continuously by automorphisms on $(X,\P,F(\P))$, then, for all $x,y \in X$, $(g,h) \mapsto c(gx,hy)$ is continuous from $G\times G$ to $E$.
 \end{lem}
 
 \begin{proof}
  Consider on the subspace $E$ of $F(\P)$ the topology given by the norm $\|.\|$ of $F(\P)$. If $X \times X$ is endowed with the product topology of $(X,d)$, as said in Remark \ref{labpart_metric_rmq}, $c: X\times X \rightarrow E$ is continuous and, since the $G$-action on $X$ is strongly continuous, for all $x,y \in X$, $(g,h) \mapsto (gx,hy)$ is continuous. Then, by composition, for all $x,y \in X$, $(g,h) \mapsto c(gx,hy)$ is continuous.
 \end{proof}

 \begin{prop}\label{labpart_labparttoaction}
  Let $G$ be a topological group acting continuously by automorphisms on a space with labelled partitions $(X,\P,F(\P))$. Then there exists a continuous isometric affine action of $G$ on a Banach subspace $B$ of $F(\P)$. \smallskip  \\
  More precisely, $B=\overline{\Vect(c(x,y) \st x,y \in X)}^{\|.\|}$ where $c$ is the separation map associated with $\P$ and $\|.\|$ is the norm of $F(\P)$, and moreover, the linear part $\pi$ and the translation part $b$ of the affine action are given by, for a fixed $x_0 \in X$: 
  $$\pi(g)\k=\k \circ\Phi_{\tau(g)} \text{ for } g \in G \text{ and } \k \in B;$$
  and
  $$b(g)=c(gx_0,x_0) \text{ for } g \in G.$$
 \end{prop}
 
 \begin{proof}
  Let $\tau$ be the $G$-action on $X$.\\
  By Definition \ref{labpart_homo} and Remark \ref{comp_homo_rmq}, the map $\Phi_{\tau(g)}:\P \rightarrow \P$ such that $\Phi_{\tau(g)}(p)=p\circ \tau(g)$ induces a linear representation $\pi$ of $G$ on $F(\P)$ given by, for $\k \in F(\P)$ and $g\in G$: 
  $$\pi(g)\k=\k\circ\Phi_{\tau(g)}.$$
  By the second requirement of Definition \ref{labpart_homo}, we have $\|\pi(g)\k\|=\|\k\|$. Thus, $\pi$ is an isometric linear representation of $G$ on $F(\P)$. \\
  Consider $E=\Vect(c(x,y)\st x,y \in X)$. Then the Banach subspace $B=\overline{E}^{\|.\|}$ of $F(\P)$ is stable under $\pi$ since $\pi(g)(c(x,y))=c(gx,gy)$ for $x,y \in X$, $g \in G$.
  Let us show that the representation $\pi$ of $G$ on $B$ is strongly continuous. Let $\k=\sum_{i=1}^{n}\l_i c(x_i,y_i) \in E$. We have, for $g \in G$, 
  $$\pi(g)\k=\k\circ \Phi_{\tau(g)}=\sum_{i=1}^{n}\l_i c(gx_i,gy_i) \in E,$$
  and, by Lemma \ref{labpart_lptoaction_lem}, for every $i$, $g \mapsto c(gx_i,gy_i)$ is continuous. \smallskip \\
  Hence, $g \mapsto \sum_{i=1}^{n}\l_i c(gx_i,gy_i)=\pi(g)\k$ is continuous. Finally, by density, for all $\k \in B$, $g \mapsto \pi(g)\k$ is continuous from $G$ to $B$. \medskip \\ 
  
  Now, let us define the translation part of the action. Fix $x_0 \in X$ and set, for all $g \in G$, $b(g)=c(gx_0,x_0) \in E$. We claim $b$ is a continuous 1-cocycle relative to $\pi$; indeed, we have, for $g \in G$, $x,y \in X$, $c(gx,gy)=c(x,y)\circ \Phi_{\tau(g)}=\pi(g)c(x,y)$ and then, for $g,h \in G$, 
  $$ b(gh)=c(ghx_0,x_0)=c(ghx_0,gx_0)+c(gx_0,x_0)=\pi(g)b(h)+b(g). $$
  The continuity of $b$ follows immediatly from Lemma \ref{labpart_lptoaction_lem}. \\
  Hence, the morphism $\a: G \rightarrow \text{Isom}(B)\cap \text{Aff}(B)$ defined by, for all $g \in G$, $\k \in B$, $\a(g)\k=\pi(g)\k+b(g)$ is a continuous isometric affine action of $G$ on $B$. 
 \end{proof}
 
 \begin{rmq}\label{labpart_labparttoaction_rmq}
  In the case where $G$ is discrete, we do not have to find a subspace of $F(\P)$ on which the representation is strongly continuous; then we have the following statement: \\
  If $G$ discrete acts by automorphisms on $(X,\P,F(\P))$, then there exists an isometric affine action of $G$ on $F(\P)$.
 \end{rmq}
  
 \begin{proof}[Proof of 2) Theorem \ref{labpart_affact}]
  Assume $G$ acts \emph{properly} continuously on a space with labelled partitions $(X,\P,F(\P))$. \\
  Consider the action $\a$ on the Banach subspace $B=\overline{E}^{\|.\|}$ given by prop \ref{labpart_labparttoaction}, where $E=\Vect(c(x,y)\st x,y \in X)$ and $\a(g)\k=\pi(g)\k+b(g)$, for $g \in G$, $\k \in B$. \smallskip \\
  Then we have, if we denote by $d$ the pseudo-metric of labelled partitions on $X$: 
  $$\|b(g)\|=\|c(gx_0,x_0)\|_{\P}=d(gx_0,x_0)\underset{g\rightarrow \infty}{\longrightarrow}\infty$$
  since the action of $G$ on $(X,\P,F(\P))$ is proper, and hence, $\a$ is a proper continuous isometric affine action of $G$ on $B$.
 \end{proof}
 
\section{Labelled partitions on a direct sum}\label{sec_dirsum}
 In this section, we define a space with labelled partitions on the direct sum of a countable family of spaces with labelled partitions and we build on it a proper action given by proper actions on each factor.  

 \subsection{Natural space with labelled partitions on a direct sum}\label{subsec_dsum}
 Given a family of space with labelled partitions, we give a natural construction of a space with labelled partitions on the direct sum of this family. A similar construction in the case of spaces with measured walls can be found in \cite{chemarval}.
 
 \begin{df}\label{const_dirsum_df}
 Let $I$ be an index set, $(X_i)_{i\in I}$ be a family of non empty sets and fix $x_0=(x_i^0)_{i \in I} \in \prod_{i \in I} X_i$. \\
 The direct sum of the family $(X_i)_{i\in I}$ relative to $x_0$ is defined by:
 $$ \lexp{x_0}{\bigoplus_{i \in I}}X_i:= \left \lbrace(x_i)_{i\in I} \in \prod_{i\in I}X_i \st x_i \neq x_i^0 \text{ for finitely many }i \in I \right\rbrace .$$ \smallskip \\
 For $i \in I$, we denote by $\pi^{X}_{X_i}:X \rightarrow X_i$ the canonical projection from the direct sum to the factor $X_i$. \medskip \\
 For $x=(x_i)_{i \in I} \in \lexp{x_0}{\bigoplus_{i \in I}}X_i$, the \emph{support} of $x$ is the finite subset of $I$: 
 $$\supp (x) = \{i \in I \st x_i \neq x_i^0 \}.$$
 \end{df}
 
 \begin{df}\label{dirsum_labpart_df}
  Let $I$ be an index set, $ \left((X_i,\P_i,F_i(\P_i))\right)_{i\in I}$ be a family of spaces with labelled partitions and fix $x_0=(x_i^0)_{i \in I} \in \prod_{i \in I} X_i$. We denote $X=\lexp{x_0}{\bigoplus_{i \in I}}X_i$. \smallskip \\
  Let $i \in I$. For $p_i \in \P_i$, we define the labelling function $p_i^{\oplus_i}: X \rightarrow \K$ by: 
  $$ p_i^{\oplus_i}= p_i\circ \pi^{X}_{X_i}.$$
  i.e., for $x=(x_i)_{i \in I} \in X$, $ p_i^{\oplus_i}(x)=p_i(x_i)$. \medskip \\
  We denote $\P_i^{\oplus_i}=\{p_i^{\oplus_i} \st p_i \in \P_i \}$, and we call the set 
  $$ \P_X=\bigcup_{i \in I} \P_i^{\oplus_i} $$
  the \emph{natural family of labelling functions on $X$} (associated with the family $( \P_i )_{i \in I}$).
 \end{df}
 
 Let $X_1,X_2$ be non empty sets and $\P_1,\P_2$ be families of labelling functions on, respectively, $X_1$ and $X_2$. \\
 In terms of partitions, if $P_1$ is the partition of $X_1$ associated with $p_1 \in \P_1$, the partition $P_1^{\oplus_1}$ of $X_1 \times X_2$ associated with $p_1^{\oplus_1}$ is: $$ P_1^{\oplus_1}=\{ h \times X_2 \st h_1 \in P_1 \}, $$
 and similarly, for $p_2 \in \P_2$, we have: $$ P_2^{\oplus_2}=\{ X_1 \times k \st k_1 \in P_2 \}. $$
 
 \begin{center}
  \includegraphics{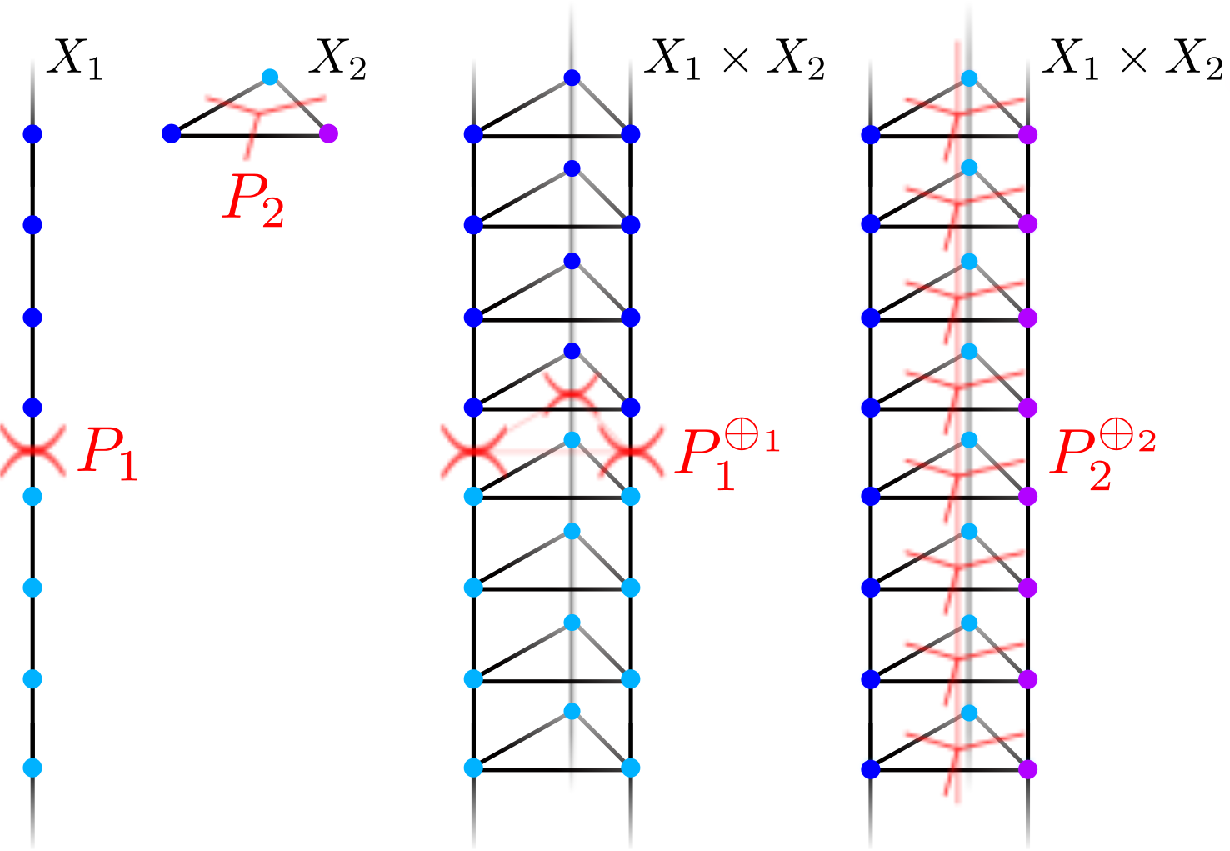} \\
  Partitions for the direct product
 \end{center}
 
 \begin{df}\label{dirsum_banach_df}
  Let $I$ be a countable index set, $ \left((X_i,\P_i,F_i(\P_i))\right)_{i\in I}$ be a family of spaces with labelled partitions and fix $x_0=(x_i^0)_{i \in I} \in \prod_{i \in I} X_i$. We denote $X=\lexp{x_0}{\bigoplus_{i \in I}}X_i$. \medskip \\
  Let $i \in I$. For $\k_i \in F_i(\P_i)$, we denote $\k_i^{\oplus_i}: \P_X \rightarrow \K$ the function:
  \begin{center}\begin{equation*}
   \k_i^{\oplus_i}(p)=\begin{cases} \k_i(p_i)\text{ if }p=p_i^{\oplus_i} \in \P_i^{\oplus_i} \\
                            0\;\;\;\;\:\:\text{ if }p=p_j^{\oplus_j} \in \P_j^{\oplus_j} \text{ with } i\neq j
                           \end{cases}
   \end{equation*}
  \end{center}
  Let $q \geq 1$. We denote $F_q(\P_X)$ the closure of 
  $$ E_q(\P_X):=\left\lbrace \sum_{i\in I} \k_i^{\oplus_i} \st \k_i \in F_i(\P_i)\text{ with }\k_i \neq 0  \text{ for a finite number of }i \in I \right\rbrace,$$ endowed with the norm $\|.\|_{q}$ defined by, for $\k=\sum_{i\in I} \k_i^{\oplus_i}$:
  $$ \|\k\|_{q}:=\left(\sum_{i \in I}\|\k_i\|_{_{F_i(\P_i)}}^q\right)^{\frac{1}{q}}.$$
  The vector space $F_q(\P_X)$ is called the \emph{$q$-space of functions on $\P_X$ of $X$}.
 \end{df}

 \begin{prop}\label{dirsum_banach}
  Let $I$ be a countable index set and $ \left((X_i,\P_i,F_i(\P_i))\right)_{i\in I}$ be a family of spaces with labelled partitions and fix $x_0=(x_i^0)_{i \in I} \in \prod_{i \in I} X_i$. We denote $X=\lexp{x_0}{\bigoplus_{i \in I}}X_i$. \medskip \\
  Then $(F_q(\P_X),\|.\|_{q})$ is isometrically isomorphic to $(\bigoplus^q_{i\in I}F_i(\P_i),\|.\|_q)$. In particular, $F_q(\P_X)$ is a Banach space.
 \end{prop}

 \subsection{Action on the natural space with labelled partitions of the direct sum}\label{subsec_actdsum}

 Let $I$ be an index set and $(H_i)_{i \in I}$ be a family of groups. We denote $e_W=(e_{H_i})_{i \in I}$ where, for $i \in I$, $e_{H_i}$ is the identity element of $H_i$. \\
 We simply denote $\displaystyle \bigoplus_{i \in I}H_i$ the group $W=\displaystyle \lexp{e_W}{\bigoplus_{i \in I}}H_i$ whose identity element is $e_W$.
 
 \begin{prop}\label{dirsum_action} $\;$ \\
  Let $I$ be a countable set and $(H_i)_{i \in I}$ be a family of groups such that, for each $i \in I$, $H_i$ acts by automorphisms on a space with labelled partitions $(X_i,\P_i,F_i(\P_i))$. We denote $X=\lexp{x_0}{\bigoplus_{i \in I}}X_i$ and $W=\bigoplus_{i \in I}H_i$. \medskip \\
  Let $q \geq 1$. Then $W$ acts by automorphisms on the natural space with labelled partitions on the direct sum $(X,\P_X,F_q(\P_X))$ via the natural action of $W$ on $X$.
 \end{prop} 
 
 \begin{proof}
  We denote by $\tau$ the $W$-action on $X$ and for $w \in W$, $p \in \P_X$, $\Phi_{\tau(w)}(p):=p\circ \tau(w)$ and, for $i \in I$, we denote by $\tau_i$ the $H_i$-action on $X$ and for $h_i \in H_i$, $p_i \in \P_i$, $\Phi_{\tau_i(h_i)}(p_i):=p_i\circ \tau_i(h_i)$.\smallskip  \\  
  Let $p \in \P_X=\bigcup_{i \in I}\P_i^{\oplus_i}$ and $w=(h_i)_{i \in I} \in W$. Then there exists $i \in I$ and $p_i \in \P_i$ such that $p=p_i^{\oplus_i}$, and we have:
  $$ \Phi_{\tau(w)}(p_i^{\oplus_i})=(\Phi_{\tau_i(h_i)}(p_i))^{\oplus_i} \in \P_i^{\oplus_i}\subset \P_X, $$
  since $\Phi_{\tau_i(h_i)}(p_i)$ belongs to $\P_i$. \medskip \\
  For $\k=\sum_{i \in I}\k_i^{\oplus_i} \in E_q(\P_X)$, we have:
  \begin{center}
   \begin{tabular}{rl}$\k\circ \Phi_{\tau(w)}(p)$&$=\k(p_i^{\oplus_i}\circ \tau(w))$ \smallskip \\
                       &$=\k((p_i\circ \tau_i(h_i))^{\oplus_i})$ \smallskip \\
                       &$=\k_i^{\oplus_i}((p_i\circ \tau_i(h_i))^{\oplus_i})$ \smallskip \\
                       &$=\k_i(p_i\circ \tau_i(h_i))$ \smallskip \\
                       &$=\k_i\circ \Phi_{\tau_i(h_i)}(p_i)$ \smallskip \\
                       $\k\circ \Phi_{\tau(w)}(p)$&$=(\k_i\circ \Phi_{\tau_i(h_i)})^{\oplus_i}(p_i^{\oplus_i}),$ \\
                      \end{tabular}
                       \end{center}
  And hence, 
  $$ \k\circ \Phi_{\tau(w)}=\sum_{i \in I}(\k_i\circ \Phi_{\tau_i(h_i)})^{\oplus_i} \in F_q(\P_X). $$
  By completeness of $F_q(\P_X)$, for all $\k \in F_q(\P_X)$, $\k\circ \Phi_{\tau(w)} \in F_q(\P_X)$. \smallskip \\
  Moreover, for $\k=\sum_{i \in I}\k_i^{\oplus_i} \in E_q(\P_X)$, we have:
  $$ \|\k\circ \Phi_{\tau(w)}\|_{q}^q=\sum_{i \in I}\|\k_i\circ \Phi_{\tau_i(h_i)}\|_{_{F_i(\P_i)}}^q=\sum_{i \in I}\|\k_i\|_{_{F_i(\P_i)}}= \|\k\|_{q}^q, $$
  since, for all $i \in I$, $\|\k_i\circ \Phi_{\tau_i(h_i)}\|_{_{F_i(\P_i)}}=\|\k_i\|_{_{F_i(\P_i)}}$. \smallskip \\
  Thus, by density of $E_q(\P_X)$ in $F_q(\P_X)$, for all $\k \in F_q(\P_X)$, $\|\k\circ \Phi_{\tau(w)}\|_{q}=\|\k\|_{q}.$ \smallskip \\
  It follows that $W$ acts by automorphisms on $(X,\P_X,F_q(\P_X))$. \smallskip \\
  
 \end{proof}
 
 When $I$ is finite, $X=\lexp{x_0}{\bigoplus_{i \in I}}X_i$ is simply the direct sum of the $X_i$ and does not depend on $x_0$. In this case, proper continuous actions on each factor $(X_i,\P_i,F_i(\P_i))$ induce a proper continuous action on the natural space with labelled partitions of the direct sum $(X,\P_X,F_q(\P_X))$:
 
 \begin{prop}\label{dirsum_finite}
  Let $n \in \N^*$. For $i \in I=\{1,...,n\}$, let $H_i$ be a topological group acting \emph{properly continuously} on a space with labelled partitions $(X_i,\P_i,F_i(\P_i))$; we denote $\displaystyle X=X_1 \times ... \times X_n$ and $W=H_1 \times ... \times H_n$. \smallskip \\
  Let $q \geq 1$. Then $W$ acts \emph{properly continuously} by automorphisms on the natural space with labelled partitions of the direct product $(X,\P_X,F_q(\P_X))$ via the natural action of $W$ on $X$.
 \end{prop}
 
 \begin{proof}
  We consider the group $W$ endowed with the product topology of the $H_i$'s. We denote by $c$ the separation map associated with $\P_X$ and, for $i \in I$, $c_i$ the separation map associated with $\P_i$.
  By Proposition \ref{dirsum_action}, $W$ acts by automorphisms on $(X,\P_X,F_q(\P_X))$. Let us show that this action is proper. For $x=(x_1,...,x_n) \in X$ and for $w=(h_1,...,h_n) \in W$, we have :
  $$ \|c(wx,x)\|_{q}^q=\sum_{i=1}^n\|c_i(h_ix_i,x_i)\|_{_{F_i(\P_i)}}^q.$$
  Thus, if $\|c(wx,x)\|_{q}\leq R$ for some $R \geq 0$, then for $i=1,...,n$, $\|c_i(h_ix_i,x_i)\|_{_{F_i(\P_i)}}\leq R$. \\
  Hence, for every $R \geq 0$ : \\ 
  $\{w=(h_i)\in W \st \|c(wx,x)\|_{q} \leq R \}$ is a subset of $\prod_{i=1}^n \{h_i \in H_i \st \|c_i(h_ix_i,x_i)\|_{_{F_i(\P_i)}}\leq R \}$ which is a relatively compact set in $W$ since each $H_i$ acts properly on $(X_i,\P_i,F_i(\P_i))$. It follows that $W$ acts properly on $(X,\P_X,F_q(\P_X))$. \\
  
  It remains to prove that the $W$-action on $(X,d)$ is strongly continuous. Remark that $d=(\sum_{i=0}^nd_i^q)^{\frac{1}{q}}$, then, the topology of $(X,d)$ is equivalent to the product topology of the $X_i$'s on $X$. \\
  Let $x=(x_i)_{i \in I} \in X$. We denote by $\tau_x:W\rightarrow X$ the function $w \mapsto wx$. For all $i \in I$, $\pi_{X_i}^X\circ \tau_x:w\rightarrow h_ix_i$ is continuous since $h_i \rightarrow h_ix_i$ is continuous; hence it follows that $\tau_x$ is continuous.
 \end{proof} \medskip

 If $I$ is countably infinite, even if each $H_i$-action on $(X_i,\P_i,F_i(\P_i))$ is proper, $W$ does not act properly on the natural space with labelled partitions on the direct sum $(X,\P_X,F_q(\P_X))$ in general. In fact, let $C$ be a positive real constant, and assume there exists, in each $H_i$, an element $h_i$ such that $\|c_i(h_ix_i^0,x_i^0)\|_{_{F_i(\P_i)}} \leq C$. For $j \in I$, the element $\d_j(h_j)$ of $W$ such that $\pi^{W}_{H_i}(\d_j(h_j))=e_{H_i}$ if $i \neq j$ and $\pi^{W}_{H_j}(\d_j(h_j))=h_j$ leaves every finite set of $W$ when $j$ leaves every finite set of $I$, but:
$$\|c(\d_j(h_j)x_0,x_0)\|_{_{F_q(\P_X)}}=\|c_i((h_j)x_i^0,x_i^0)\|_{_{F_i(\P_i)}} \leq C.$$
And then, $W$ does not act properly on $(X,\P_X,F_q(\P_X))$. \smallskip \\
To make $W$ act properly on a space with labelled partitions in the case where $W$ is endowed with the discrete topology, we have to define a structure of labelled partitions on $W$ such that the labelled partitions metric between $e_W$ and $w$ goes to infinity when the support of $w$ leaves every finite set in $I$. To build this structure, we scale every labelling function of the naive $\ell^q$ space with labelled partitions on each factor $H_i$ by a weight depending on $i$ which grows as $i$ leaves every finite set in $I$. \medskip \\

 \begin{nt}\label{const_suppinf}
 Let $I$ be a countable index set and $X=\lexp{x_0}{\bigoplus_{i \in I}}X_i$ be a direct sum of sets $X_i$'s. \medskip \\
 We say that, for $x \in X$, \emph{$supp(x)$ leaves every finite set in $I$} or \emph{$supp(x)\rightarrow \infty$ in $I$} if there exists $j \in supp(x)$ which leaves every finite set in $I$.
\end{nt} 

% \begin{prop}\label{const_dirsuminfty}
%  Let $I$ be a countable index set and $X=\lexp{x_0}{\bigoplus_{i \in I}}X_i$ be a direct sum of countable sets $X_i$'s.
%  Then, an element $x=(x_i)_{i\in I} \in X$ leaves every finite set in $X$ if either there exists $j \in I$ such that $x_j$ leaves every finite set in $X_j$ or $supp(x)$ leaves every finite set in $I$.
% \end{prop}
 
 \begin{df}\label{phinaivelabpart}
  Let $X$ be a set and $w$ be a non-negative real. \\
  We set, for $x \in X$:
  $$\lexp{(w)}{\di}_{x}:=2^{-\frac{1}{q}}w\d_{x}:X\rightarrow \K,$$
  where $\d_x:X\rightarrow \{0,1\}$ is the Dirac function at $x$, and we call the set
  $$\lexp{(w)}{\di}:=\{\lexp{(w)}{\di}_{x} \st x \in X \},$$
  the \emph{$w$-weighted naive family of labelling functions} on $X$.
 \end{df}
 
 \begin{prop}
  Let $X$ be a set and $w$ be a non-negative real. \smallskip \\
  Let $q \geq 1$. Then the triple $(X,\lexp{(w)}{\di},\ell^q(\lexp{(w)}{\di}))$ is a space with labelled partitions. \\
  Moreover, if a group $H$ acts on $X$, then $H$ acts by automorphisms on $(X,\lexp{(w)}{\di},\ell^q(\lexp{(w)}{\di}))$.
 \end{prop}
 
 \begin{proof}
  It is a straightfoward generalization of Proposition \ref{naivelabpart_prop} and Remark \ref{labpart_homo_ex}.
 \end{proof}
 
 Subsquently, for a countably infinite set $I$, we consider a function $\phi: I \rightarrow \R_+$ such that $\phi(i)\underset{i\rightarrow \infty}{\longrightarrow}+\infty$ (such a function always exists when $I$ is countably infinite: for instance, take any bijective enumeration function $\phi$ from $I$ to $\N$).

 \begin{lem}\label{propdirsum_lem}
  Let $I$ be a countably infinite set and $(H_i)_{i \in I}$ be a family of countable discrete groups and we denote $W$ the group $\bigoplus_{i \in I}H_i$ endowed with the discrete topology.
  Consider, on each $H_i$, the $\phi(i)$-weighted naive family of labelling functions $\lexp{(\phi(i))}{\di}$ and we denote by $\lexp{(\phi)}{\di}=\bigcup_{i \in I}\lexp{(\phi(i))}{\di}^{\oplus_i}$ the natural set of labelling functions associated with $(\lexp{(\phi(i))}{\di})_{i \in I}$.\medskip \\
  Let $q \geq 1$. Then, $W$ acts by automorphisms on the natural space with labelled partitions on the direct sum $(W,\lexp{(\phi)}{\di},F_q(\lexp{(\phi)}{\di}))$. \\
  Moreover, we have:
  $$ \|c_{\phi}(w,e_W)\|_{_{F_q(\lexp{(\phi)}{\di})}}\rightarrow +\infty \text{ when } \supp(w) \rightarrow \infty \text{ in } I, $$
  where $c_{\phi}$ is the separation map associated with $\lexp{(\phi)}{\di}$.
 \end{lem}
 
 \begin{proof}
  By Proposition \ref{dirsum_action}, $W$ acts by automorphisms on $(W,\lexp{(\phi)}{\di},F_q(\lexp{(\phi)}{\di}))$ and we have, for $w=(h_i),w'=(h'_i) \in W$:
  \begin{center}
   \begin{tabular}{rl}
    $\|c_{\phi}(w,w')\|_{_{F_q(\lexp{(\phi)}{\di})}}^q$&$\displaystyle =\sum_{i\in I}\|c_{\phi(i)}(h_i,h'_i)\|_q^q$ \smallskip \\
    &$\displaystyle =\sum_{i\in \supp(w^{-1}w')}\phi(i)^q$.  \\
   \end{tabular}
  \end{center}
 Let $w \in W$ such that $\supp(w) \rightarrow \infty$ in $I$. Then there exists $j \in \supp(w)$ such that $j\rightarrow \infty$ in $I$ and hence:
 $$ \|c_{\phi}(w,e_W)\|_{_{F_q(\lexp{(\phi)}{\di})}}^q=\sum_{i\in \supp(w)}\phi(i)^q \geq \phi(j)^q \rightarrow +\infty. $$

 \end{proof}

\begin{prop}\label{const_dirsumproper}
 Let $I$ be a countably infinite set and $(H_i)_{i \in I}$ be a family of countable discrete groups such that, for each $i \in I$, $H_i$ acts \emph{properly} by automorphisms on a space with labelled partitions $(X_i,\P_i,F_i(\P_i))$. We denote $X=\lexp{x_0}{\bigoplus_{i \in I}}X_i$ and $W=\bigoplus_{i \in I}H_i$ endowed with the discrete topology. \smallskip \\
 Let $q \geq 1$. Then there exists a structure of space with labelled partitions $(Y,\P_Y,F_q(\P_Y))$ on which $W$ acts \emph{properly} by automorphisms. \smallskip \\
 More precisely, $(Y,\P_Y,F(\P_Y))$ is the natural space with labelled partitions on the direct product $Y=X\times W$ where : \smallskip 
 \begin{itemize}
  \item on $X$, we consider the natural space with labelled partitions on the direct sum of the family $((X_i,\P_i,F_i(\P_i)))_{i \in I}$; \smallskip 
  \item on $W$, we consider the natural space with labelled partitions on the direct sum of the family $((H_i,\lexp{(\phi(i))}{\di},\ell^q(\lexp{(\phi(i))}{\di})))_{i \in I}$ where for $i \in I$, $\lexp{(\phi(i))}{\di}$ is the $\phi(i)$-weighted naive family of labelling functions on $H_i$.
 \end{itemize}
\end{prop}

\begin{proof}
 By Proposition \ref{dirsum_action}, $W$ acts by automorphisms on both $(X,\P_X,F_q(\P_X))$ and $(W,\lexp{(\phi)}{\di},\ell^q(\lexp{(\phi)}{\di}))$.
 We set $Y=X \times W$ and consider the natural space with labelled partitions $(Y,\P_Y,F_q(\P_Y))$ on the direct product where:
 $$ \P= \P_X^{\oplus_1}\cup \lexp{(\phi)}{\di}^{\oplus_2}, $$
 and 
 $$ F_q(\P) \simeq F_q(\P_X) \oplus \ell^q(\lexp{(\phi)}{\di}).$$
 Then, by Proposition \ref{dirsum_action}, $W\times W$ acts by automorphisms on $(Y,\P_Y,F_q(\P_Y))$ via the action $(w_1,w_2).(x,w)=(w_1.x,w_2w)$.
 Hence, $W$ acts by automorphisms on $(Y,\P_Y,F_q(\P_Y))$, where $W$ is viewed as the diagonal subgroup $\{(w,w)\st w \in W\} < W\times W$. \smallskip \\
 
 It remains to prove that the $W$-action on $(Y,\P_Y,F_q(\P_Y))$ is proper. We have, for $w=(h_i) \in W$:
 \begin{center}
   \begin{tabular}{rl}
    $\|c_{\P_Y}(w.(x_0,e_W),(x_0,e_W))\|_{_{F_q(\P_Y)}}^q$&$\displaystyle =\|c_{\P_X}(w.x_0,x_0)\|_{_{F_q(\P_X)}}^q+\|c_{\phi}(w,e_W)\|_q^q$ \smallskip \\
    &$\displaystyle =\sum_{i\in \supp(w)}\|c(h_ix^0_i,x^0_i)\|_{_{F_i(\P_i)}}^q+\sum_{i\in \supp(w)}\phi(i)^q$.  \\
   \end{tabular}
  \end{center}
Hence, for $R \geq 0$, $\|c_{\P_Y}(w.(x_0,e_W),(x_0,e_W))\|_{_{F_q(\P_Y)}} \leq R$ implies that $\|c(h_ix^0_i,x^0_i)\|_{_{F_i(\P_i)}} \leq R$ and $\phi(i) \leq R$ for all $i \in \supp(w)$. Thus, for all $R \geq 0$, $\{w \st \|c_{\P_Y}(w.(x_0,e_W),(x_0,e_W))\|_{_{F_q(\P_Y)}} \leq R \}$ is a subset of 
$$ \left\lbrace w=(h_i) \st \supp(w) \subset \{j \st \phi(j) \leq R\} \text{ and } \{ h_i \in H_i \st \|c(h_ix^0_i,x^0_i)\|_{_{F_i(\P_i)}} \leq R \}\right\rbrace, $$
which is a finite set as $\{j \st \phi(j) \leq R\}$ is finite and by properness of the $H_i$'s actions, each $ \{ h_i \in H_i \st \|c(h_ix^0_i,x^0_i)\|_{_{F_i(\P_i)}}\leq R\}$ is finite. \\
% If there exists $j \in I$, such that $h_j \rightarrow \infty$ in $H_j$, then, since the $H_j$-action is proper, we have: 
% $$\|c_{\P_Y}(w.(x_0,e_W),(x_0,e_W))\|_{_{F_q(\P_Y)}} \geq \sum_{i\in \supp(w)}\|c(h_ix^0_i,x^0_i)\|_{_{F_i(\P_i)}} \rightarrow +\infty,$$
% and if $supp(w)\rightarrow \infty$ in $I$, by Lemma \ref{propdirsum_lem}, we have: 
% $$\|c_{\P_Y}(w.(x_0,e_W),(x_0,e_W))\|_{_{F_q(\P_Y)}} \geq \|c_{\phi}(w,e_W)\|_q \rightarrow +\infty.$$
It follows that $W$ acts properly on $(Y,\P_Y,F_q(\P_Y))$.
\end{proof}

 \subsection{Action of a semi-direct product on a space with labelled partitions}\label{subsec_semidir}

%  \begin{df}[compatible action]\label{const_semidir_comp}
%   Let $G_1,G_2$ be groups and $\rho:G_2\rightarrow Aut(G_1)$ be a morphism of groups. \\
%   Consider a set $X$ on which $G_1$ acts by $\a_1$ and $G_2$ acts by $\a_2$. We say $\a_2$ is \emph{compatible} with $\a_1$ with respect to $\rho$ if, for all $g_1 \in G_1$, $g_2 \in G_2$, we have:
%   $$ \a_2(g_2)\circ \a_1(g_1)=\a_1(\rho(g_2)g_1)\circ \a_2(g_2). $$
%  \end{df} 

\begin{df}[compatible action]\label{const_semidir_comp}
  Let $G_1,G_2$ be groups and $\rho:G_2\rightarrow Aut(G_1)$ be a morphism of groups. \\
  Consider a set $X$ on which $G_1$ and $G_2$. We say that the $G_2$-action is \emph{compatible} with the $G_1$-action with respect to $\rho$ if, for $g_1 \in G_1$, $g_2 \in G_2$, we have, for all $x \in X$ :
  $$ g_2g_1g_2^{-1}x=\rho(g_2)(g_1)x. $$
 \end{df}

 \begin{exemple}\label{const_semidir_comp_ex}
  If $\rho:G_2\rightarrow Aut(G_1)$ is a morphism, then the action $\rho$ of $G_2$ on $G_1$ is compatible with the action of $G_1$ on itself by translation with respect to $\rho$. 
 \end{exemple}
 \bigskip

\begin{prop}\label{const_semidir_labpart}
  Let $(X_1,\P_1,F_1(\P_1))$,$(X_2,\P_2,F_2(\P_2))$ be spaces with labelled partitions and $G_1,G_2$ be topological groups acting continuously by automorphisms on, respectively,\linebreak $(X_1,\P_1,F_1(\P_1))$ and $(X_2,\P_2,F_2(\P_2))$ via $\tau_1$ and $\tau_2$. \\ Let $ \rho:G_2\rightarrow Aut(G_1)$ be a morphism of groups such that $(g_1,g_2) \mapsto \rho(g_2)g_1$ is continuous for the product topology on $G_1 \times G_2$. \smallskip \\
  Assume that there exists a continuous action by automorphisms of $G_1\rtimes_{\rho}G_2$ on $X_1$ which extends the $G_1$ action. \medskip \\
  Then the semi-direct product $G_1\rtimes_{\rho}G_2$ acts continuously by automorphisms on the natural structure of labelled partitions $(X_1\times X_2, \P,F_q(\P))$ on the direct product of $X_1\times X_2$. \smallskip \\
  Moreover, if, for $i=1,2$, $G_i$ acts properly on $(X_i,\P_i,F_i(\P_i))$, then $G_1\rtimes_{\rho}G_2$ acts properly on $(X_1\times X_2, \P,F(\P))$.
 \end{prop}

 \begin{proof}
  Let us denote by $\tau_1$ the $G_1$-action on $X_1$, by $\tau_2$ the $G_2$-action on $X_2$ and by $\tilde{\rho}$ the $G_2$-action on $G_1$ defined by the restriction on $G_2$ of the $G_1\rtimes_{\rho}G_2$-action on $X_1$. Then $\tilde{\rho}$ is compatible with $\tau_1$ with respect to $\rho$. \\  
  We denote by $\tau$ the action of $G=G_1\rtimes_{\rho}G_2$ on $X=X_1\times X_2$ defined by:
  $$\tau(g_1,g_2)(x_1,x_2)=(\tau_1(g_1)(\tilde{\rho}(g_2)x_1),\tau_2(g_2)x_2).$$ 
  We show that, via this action, $G$ acts by automorphisms on the direct product of spaces with labelled partitions $(X, \P,F_q(\P))$ where $\P=\P_1^{\oplus_1} \cup \P_2^{\oplus_2}$ and $F_q(\P)\simeq F_1(\P_1)\oplus F_2(\P_2)$ endowed with the $q$-norm of the direct sum for $q \geq 1$. \\
  
  Let $p\in \P$ and $g=(g_1,g_2) \in G$. If $p=p_1^{\oplus_1} \in \P_1^{\oplus_1}$, then, for all $x=(x_1,x_2) \in X$, we have:
  \begin{center}

  \begin{tabular}{rl}$\Phi_{\tau(g)}(p)(x)$&$=p(\tau(g)x)$ \smallskip \\
                       &$=p_1^{\oplus_1}(\tau_1(g_1)(\tilde{\rho}(g_2)x_1),\tau_2(g_2)x_2)$ \smallskip \\
                       &$=p_1(\tau_1(g_1)(\tilde{\rho}(g_2)x_1))$ \smallskip \\
                       &$=p_1\circ \tau_1(g_1)\circ \tilde{\rho}(g_2)(x_1)$ \smallskip \\
                       $\Phi_{\tau(g)}(p)(x)$&$=(p_1\circ \tau_1(g_1)\circ \tilde{\rho}(g_2))^{\oplus_1}(x_1,x_2),$ \medskip \\
  \end{tabular}
 \end{center}
 and since $G_1$ acts by automorphisms on $(X_1,\P_1,F_1(\P_1))$ via $\tau_1$, we have $p_1\circ \tau_1(g_1) \in \P_1$, and $G_2$ acts by automorphisms on $(X_1,\P_1,F_1(\P_1))$ via $\tilde{\rho}$, then $p_1\circ \tau_1(g_1) \circ \tilde{\rho}(g_2) \in \P_1$. \\
 Hence, $\Phi_{\tau(g)}(p)=(p_1\circ \tau_1(g_1)\circ \tilde{\rho}(g_2))^{\oplus_1}$ belongs to $\P$. \smallskip \\
 For $p=p_2^{\oplus_2} \in \P_2^{\oplus_2}$, we have $\Phi_{\tau(g)}(p)=(p_2\circ \tau_2(g_2))^{\oplus_2}$ which belongs to $\P$ since $G_2$ acts by automorphisms on $(X_2,\P_2,F_2(\P_2))$ via $\tau_2$. \\
 Then, for all $g \in G$ and all $p \in \P$, $$ \Phi_{\tau(g)}(p)=p \circ \tau(g) \in \P. $$
 Let us fix some notations. We denote, for $g_1 \in G_1$, $g_2 \in G_2$:\medskip  \\
 - $\Phi^{^{(1)}}_{\tau_1(g_1)}: \P_1 \rightarrow \P_1$ the map $\Phi^{^{(1)}}_{\tau_1(g_1)}(p_1)=p_1\circ \tau_1(g_1)$; \smallskip \\
 - $\Phi^{^{(\tilde{\rho})}}_{\tilde{\rho}(g_2)}: \P_1 \rightarrow \P_1$ the map $\Phi^{^{(\tilde{\rho})}}_{\tilde{\rho}(g_2)}(p_1)=p_1\circ \tilde{\rho}(g_2)$; \smallskip \\
 - $\Phi^{^{(2)}}_{\tau_2(g_2)}: \P_2 \rightarrow \P_2$ the map $\Phi^{^{(2)}}_{\tau_2(g_2)}(p_2)=p_2\circ \tau_2(g_2)$. \medskip \\
 Let $\k$ be in $F(\P)$ and $g=(g_1,g_2) \in G$. We have, for all  $p_1 \in \P_1$ and all $p_2 \in \P_2$: 
 $$\k \circ \Phi_{\tau(g)}(p_1^{\oplus_1})=(\k_1\circ\Phi^{^{(\tilde{\rho})}}_{\tilde{\rho}(g_2)} \circ \Phi^{^{(1)}}_{\tau_1(g_1)})^{\oplus_1}(p_1^{\oplus_1}),$$ and $$\k \circ \Phi_{\tau(g)}(p_2^{\oplus_2})=(\k_1\circ \Phi^{^{(2)}}_{\tau_2(g_2)})^{\oplus_2}(p_2^{\oplus_2}). $$
 Hence, $\k \circ \Phi_{\tau(g)} = (\k_1\circ \Phi^{^{(\tilde{\rho})}}_{\tilde{\rho}(g_2)} \circ \Phi^{^{(1)}}_{\tau_1(g_1)})^{\oplus_1} + (\k_2\circ \Phi^{^{(2)}}_{\tau_2(g_2)})^{\oplus_2}$ and we have: \medskip \\
 
 \begin{tabular}{rl}$\|\k \circ \Phi_{\tau(g)} \|_{q}^q$&$=\|\k_1\circ \Phi^{^{(\tilde{\rho})}}_{\tilde{\rho}(g_2)} \circ \Phi^{^{(1)}}_{\tau_1(g_1)}\|_{_{F_1(\P_1)}}^q+\|\k_2\circ \Phi^{^{(2)}}_{\tau_2(g_2)}\|_{_{F_2(\P_2)}}^q$ \smallskip \\
                       &$=\|\k_1\|_{_{F_1(\P_1)}}^q+\|\k_2\||_{_{F_2(\P_2)}}^q$ \smallskip \\
                       $\|\k \circ \Phi_{\tau(g)} \|_{q}$&$=\|\k\|_{q}$ \smallskip \\
                      \end{tabular} \\
It follows that $G_1\rtimes_{\rho}G_2$ acts by automorphisms on the space with labelled partitions $(X_1\times X_2,\P,F_q(\P))$. \medskip \\
It remains to check this action by automorphisms is continuous, i.e. for all $x \in X$, $g \mapsto \tau(g)x$ is continuous. \smallskip \\
As a set $G_1\rtimes_{\rho}G_2$ is simply $G_1\times G_2$ and since $(g_1,g_2) \mapsto \rho(g_2)g_1$ is continuous, the product topology on $G_1\times G_2$ is compatible with the group structure of $G_1\rtimes_{\rho}G_2$ (see \cite{boutop}, III.18 Proposition 20). \\
Moreover, $\tau_1$, $\tau_2$ and $\tilde{\rho}$ are strongly continuous, then, for all $(x_1,x_2) \in X$, the map $(g_1,g_2) \rightarrow (\tau(g_1)(\tilde{\rho}(g_2)x_1),\tau_2(g_2)x_2)$ is continuous from $G_1 \times G_2$ endowed with the product topology to $(X,d)$ where $d$ is the labelled partitions pseudo-metric. \\
Hence, $G_1\rtimes_{\rho}G_2$ acts continuously by automorphisms on $(X,\P,F_q(\P))$. \bigskip \\
Assume, for $i=1,2$, $G_i$ acts properly on $(X_i,\P_i,F_i(\P_i))$ via $\tau_i$, and we denote by $c_i$ the separation map associated with $\P_i$. \\
%Since $N$ is a norm on $\R^2$ and all norms on $\R^2$ are equivalent, we have: if $z_1 \rightarrow +\infty$ or $z_2 \rightarrow +\infty$ then $N(z_1,z_2)\rightarrow +\infty$.
Fix $x_0=(x_1,x_2) \in X_1\times X_2$. \\
The following egality holds for every $g=(g_1,g_2) \in G_1\rtimes_{\rho}G_2$: 
$$\|c(\tau(g)x_0,x_0)\|_{q}^q=\|c_1(\tau_1(g_1)(\tilde{\rho}(g_2)x_1),x_1)\|_{_{F_1(\P_1)}}^q+\|c_2(\tau_2(g_2)x_2,x_2)\|_{_{F_2(\P_2)}}^q .$$
Since $G_1\rtimes_{\rho}G_2$ is endowed with the product topology of $G_1$ and $G_2$, $g=(g_1,g_2) \rightarrow \infty$ in $G_1\rtimes_{\rho}G_2$  if, and only if, $g_1 \rightarrow \infty$ in $G_1$ or $g_2 \rightarrow \infty$ in $G_2$. Hence, we have two disjoint cases: \smallskip \\
First case: $g_1 \rightarrow \infty$ in $G_1$ and $g_2$ belongs to a compact subset $K_2$ of $G_2$. \\
By continuity of $g'_2 \mapsto \|c(\tilde{\rho}(g'_2)x_1,x_1)\|_{_{F_1(\P_1)}}$, there exists $C(K_2) \geq 0$ such that, for every $g'_2 \in K_2$, $\|c(\tilde{\rho}(g'_2)x_1,x_1)\|_{_{F_1(\P_1)}}\leq C(K_2)$, and, hence, 
\begin{center}
 \begin{tabular}{rl}
  $\|c(\tau(g_1)\tilde{\rho}(g_2)x_1,\tilde{\rho}(g_2)x_1)\|_{_{F_1(\P_1)}}$&$ \leq \|c(\tau_1(g_1)\tilde{\rho}(g_2)x_1,x_1)\|_{_{F_1(\P_1)}}+\|c(\tilde{\rho}(g_2)x_1,x_1)\|_{_{F_1(\P_1)}}$ \smallskip\\
  &$\leq \|c(\tau(g_1)\tilde{\rho}(g_2)x_1,x_1)\|_{_{F_1(\P_1)}} +C(K_2)$.
 \end{tabular}
\end{center}
But, since $G_1$ acts properly on $(X_1,\P_1,F_1(\P_1))$, $\|c(\tau(g_1)\tilde{\rho}(g_2)x_1,\tilde{\rho}(g_2)x_1)\|_{_{F_1(\P_1)}}\underset{g_1\rightarrow \infty}{\longrightarrow} +\infty$, and then,
 $$ \|c(\tau(g_1)\tilde{\rho}(g_2)x_1,x_1)\|_{_{F_1(\P_1)}}\underset{g_1\rightarrow \infty}{\longrightarrow} +\infty. $$
It follows that $\|c(\tau(g)x_0,x_0)\|_{q} \underset{g_1\rightarrow \infty}{\longrightarrow} +\infty$. \medskip \\
Second case: $g_2 \rightarrow \infty$ in $G_2$. \\
We have $\|c_2(\tau_2(g_2)x_2,x_2)\|_{_{F_2(\P_2)}}\underset{g_2\rightarrow \infty}{\longrightarrow} +\infty$ and then $\|c(\tau(g)x_0,x_0)\|_{q} \rightarrow +\infty$. \smallskip \\
Finally, as required, we have $$\|c(\tau(g)x_0,x_0)\|_{q} \underset{g\rightarrow \infty}{\longrightarrow} +\infty,$$
and then, $G_1\rtimes_{\rho} G_2$ acts properly by automorphisms on $(X,\P,F_q(\P))$.

 \end{proof}
 
\section{Wreath products and property $PL^p$}\label{sec_wreath}

Using Proposition \ref{const_semidir_labpart}, we simplify a part of the proof of Th 6.2 in \cite{wreath} where Cornulier, Stalder and Valette establish the stability of the Haagerup property by wreath product; and we generalize it in the following way: the wreath product of a group with property $PL^p$ by a Haagerup group has property $PL^p$.

\begin{noth}[\ref{const_semidir_wreathflp}]
  Let $H,G$ be countable discrete groups, $L$ be a subgroup of $G$ and $p > 1$, with $p \notin 2\Z\smallsetminus \{2\}$. We denote by $I$ the quotient $G/L$ and $W=\bigoplus_{I}H$. Assume that $G$ is Haagerup, $L$ is co-Haagerup in $G$ and $H$ has property $PL^p$. \\
  Then the permutational wreath product $H\wr_I G=W\rtimes G$ has property $PL^p$.
 \end{noth}

\subsection{Permutational wreath product}

We first introduce the notion of permutational wreath product :

\begin{df}\label{wreath_df}
 Let $H,G$ be countable groups, $I$ be a $G$-set and $W=\bigoplus_{i \in I}H$. The \emph{permutational wreath product} $H\wr_I G$ is the group:
 $$ H\wr_I G:=W\rtimes_{\rho} G, $$
 where $G$ acts by shift on $W$ via $\rho$ i.e. $\rho(g): (h_i)_{i \in I} \mapsto (h_{g^{-1}i})_{i \in I}$, for $g \in G$. \smallskip \\
 When $I=G$, $H\wr_G G$ is simply called \emph{wreath product} and is denoted $H\wr G$.
\end{df}

%résultats sur la propriété (T) et haagerup pour les pdts en couronne +résultat de Ioana-Chifan.

\subsection{Property $PL^p$ for the permutational wreath product}

 To prove Theorem \ref{const_semidir_wreathflp}, we need the following structure of space with measured walls relative to the wreath product built in \cite{wreath}, Theorem 4.2 (see \cite{wreath} $\S$ 6.1 for examples of co-Haagerup subgroups) : 
 
 \begin{df}
  Let $G$ be a group and $L$ be subgroup of $G$. We say that $L$ is \emph{co-Haagerup} in $G$ if there exists a proper $G$-invariant conditionally negative definite kernel on $G/L$.
 \end{df}

 \begin{theo}[Cornulier, Stalder, Valette]\label{const_semidir_theoCSV}
  Let $H,G$ be countable discrete groups and let $L$ be a subgroup of $G$. We denote by $I$ the quotient $G/L$ and $W=\bigoplus_{I}H$. \\
  Suppose that $G$ is Haagerup and that $L$ is co-Haagerup in $G$. \\
  Then there exists a structure $(W\times I,\mu)$ of space with measured walls on $W\times I$, with wall pseudo-metric denoted by $d_{\mu}$, on which $W\rtimes G$ acts by automorphisms and which satisfies, for any $x_0=(w_0, i_0) \in W \times I$ and for all $g \in G$: \\
  $$ d_{\mu}((w,g)x_0,x_0)\rightarrow +\infty \text{ when }w \in W \text{ is such that } \supp (w) \rightarrow \infty \text{ in } I. $$
 \end{theo}
 
 \begin{rmq}\label{rmk_CSV}
  The property ``$ d_{\mu}((w,g)x_0,x_0)\rightarrow +\infty$ when $w \in W$ is such that $\supp (w) \rightarrow \infty$ in $I$`` can be reformulated as follows : \\
  For all $R \geq 0$, there exists a finite set $J_R \subset I$ such that, for $(w,g) \in H\wr_I G$, $$d_{\mu}((w,g)x_0,x_0) \leq R \text{ implies } \supp(w) \subset J_R.$$
 \end{rmq}

 \begin{lem}\label{const_semidir_wreathflp_lem}
  Let $H,G$ be countable discrete groups, $L$ be a subgroup of $G$ and $q \geq 1$. We denote by $I$ the quotient $G/L$ and $W=\bigoplus_{I}H$. Suppose that $G$ is Haagerup, $L$ is co-Haagerup in $G$ and $H$ has property $PL^q$. \\
  Then $W$ and $G$ acts by automorphisms on a space $(X,\P,F(\P))$ with labelled partitions such that:
  \begin{itemize}
   \item the $W$-action is proper, \\
   \item the $G$-action is compatible with the $W$-action, \\
   \item the Banach space $F(\P)$ is isometrically isomorphic to a Banach subspace of a $L^q$ space. 
  \end{itemize}
 \end{lem}
 
 \begin{proof}
  Consider the $W\rtimes G$-action on the space with measured walls $(W\times I,\mu)$ given by Proposition \ref{const_semidir_theoCSV}. Then, by Proposition \ref{walls_labpart}, $W\rtimes G$ acts by automorphisms on the space with labelled partitions $(W\times I,\P_{\mu},L^q(\P_{\mu},\mu))$. Let $y_0=(e_W,i_0) \in W \times I$. The separation map $c_{\mu}$ associated with $\P_{\mu}$ satisfies: 
  $$\|c_{\mu}((w,g)y_0,y_0)\|_q^q=d_{\mu}((w,g)y_0,y_0).$$
  Now, consider the structure of space with labelled partitions on $H$ given by its proper isometric affine action on a space $L^q(E,\nu)$. By Proposition \ref{dirsum_action}, $W$ acts by automorphisms on the natural structure of space with labelled partitions $(W,\P_W,F_q(\P_W))$ of the direct sum of spaces with labelled partitions on $H$. Moreover, $G$ acts by automorphisms on $(W,\P_W,F_q(\P_W))$ by shift via its action on $I$. \smallskip \\
  We denote $X=(W\times I) \times W$ and consider the space with labelled partitions $(X,\P,F(\P))$ given by the direct product of spaces with labelled partitions $(W\times I,\P_{\mu},L^q(\P_{\mu},\mu))$ and $(W,\P_W,F_q(\P_W))$. Then we have actions by automorphisms $\tau_W$ of $W$ and $\tau_G$ on $X$ given by, for $x=(w_1,i,w_2) \in X$, $w \in W$ and $g \in G$:
  $$ \tau_W(w)x=(ww_1,i,ww_2) \text{ and } \tau_G(g)x=(\rho(g)w_1,gi,\rho(g)w_2). $$
  The action $\tau_G$ is clearly compatible with $\tau_W$ since $W\rtimes_{\rho} G$ acts naturally on $W$ and on $W\times I$. \medskip \\
  The Banach space $F(\P)$ is isometrically isomorphic to the $q$-direct sum $L^q(\P_{\mu},\mu)\oplus F_q(\P_W) $, then $F(\P)$ is isometrically isomorphic to a Banach subspace of $ L^q(\P_{\mu},\mu) \oplus (\bigoplus_I^q L^q(E,\nu))$. It follows that $F(\P)$ is isometrically isomorphic to a Banach subspace of a $L^q$ space.
  We denote $x_0=(e_W,i_0,e_W) \in X$. We have, for $w=(h_i)_{i \in I} \in W$: \medskip \\
  \begin{tabular}{rl}
   $\|c(\tau_W(w)x_0,x_0)\|_{_{F(\P)}}^q$&$=\|c_{\mu}((w,i_0),(e_W,i_0))\|_q^q+\|c_{\P_W}(w,e_W)\|_{_{F_{q}(\P_{W})}}^q$ \smallskip \\
   &$\displaystyle =d_{\mu}((w,e_G)i_0,i_0)+\sum_{i \in \supp(w)}\|c_{\P_H}(h_i,e_H)\|_{_{F_H(\P_H)}}^q$ \\
  \end{tabular} \medskip \\
  Hence, $W$ acts properly by automorphisms on $(X,\P,F(\P))$: indeed, for $R \geq 0$ and $w=(h_i) \in W$, $\|c(\tau_W(w)x_0,x_0)\|_{_{F(\P)}} \leq R$ implies $d_{\mu}((w,e_G)i_0,i_0) \leq R^q$ and $\|c_{\P_H}(h_i,e_H)\|_{_{F_H(\P_H)}} \leq R$. 
  It follows that, for $R \geq 0$ and $J_{R^q} \subset I$ as in Remark \ref{rmk_CSV}, $\{ w \in W \st \|c(\tau_W(w)x_0,x_0)\|_{_{F(\P)}} \leq R \}$ is a subset of :
  $$ \left\lbrace w=(h_i) \st \supp(w) \subset J_{R^q} \text{ and } \{ h_i \in H \st \|c_{\P_H}(h_i,e_H)\|_{_{F_H(\P_H)}} \leq R \}\right\rbrace, $$
  which is a finite set as $J_{R^q}$ is a finite set and $H$-action is proper.

%   $w=(h_i) \rightarrow \infty$ in $W$ if, and only if, $\supp(w) \rightarrow \infty$ in $I$ or there exists $j \in I$ such that $h_j \rightarrow \infty$ in $H$; then, in the first case, by the previous theorem, $d_{\mu}((w,e_G)y_0,y_0) \rightarrow +\infty$ and in the second case, $\sum_{i \in \supp(w)}\|c_{\P_H}(h_i,e_H)\|_{_{F_H(\P_H)}}^q\geq \|c_{\P_H}(h_j,e_H)\|_{_{F_H(\P_H)}}^q \rightarrow +\infty$. \medskip \\
 \end{proof}

 \begin{proof}[Proof of Theorem \ref{const_semidir_wreathflp}]
  By Lemma \ref{const_semidir_wreathflp_lem}, $W$ and $G$ act by automorphisms on a space $(X,\P,F(\P))$ with labelled partitions such that the $W$-action is proper, and the $G$-action is compatible with the $W$-action with respect to $\rho$. Moreover, since $G$ is Haagerup, $G$ acts properly by automorphisms on a space $(Y,\P',F'(\P'))$ with labelled partitions where $F'(\P')$ isometrically isomorphic to a $L^q$ space. \\
  Hence, by Theorem \ref{const_semidir_labpart}, $H\wr_I G=W\rtimes_{\rho} G$ acts properly by automorphisms on a space $(Z,\P_Z,F_Z(\P_Z))$ where $F_Z(\P_Z)$ is isometrically isomorphic to $F(\P) \oplus F'(\P')$ endowed with the $q$-norm of the direct sum.
  It follows that $F_Z(\P_Z)$ is isometrically isomorphic to a Banach subspace of a $L^q$ space. \\
  Thus, by Corollary \ref{labpart_aflp}, $H\wr_I G$ has property $PL^q$.
 \end{proof}

 \section{Amalgamated free product}\label{freepartsec}
 
 In this section, we develop tools around the notion of tree of spaces in order to build a structure of space with labelled partitions on which an amalagmated free product acts by automorphisms given actions of the factors on some spaces with labelled partitions. \\
 
 \subsection{Labelled partitions on a tree of spaces}\label{free_tree_sec}
 
 A tree is a pair of sets $T=(V,\E)$, where $V$ is the set of vertices and $\E$ is the set of edges, together with an injective map $\E \rightarrow \{ \{v,w\} \st v\neq w \in V \}$; and satisfies that every two vertices are connected by a unique edge path, that is, a path without backtracking. \\
 The set of vertices $V$ can be endowed with a natural metric $d_T$ : the distance between two vertices is the number of edges in the edge path joining them. Moreover, we say that a vertex $u$ is between $v$ and $w$ in $V$ if $u$ is an endpoint of some edge which belongs to the edge path between $v$ and $w$.

 \begin{df}\label{free_tree_df}
 Let $T=(V,\E)$ be a tree, $\left(X_v\right)_{v \in V}$ and $\left(X_e\right)_{e \in \E}$ be collections of non empty sets such that there exists, for all $e=\{v,w\} \in \E$ : 
 $$ \s_{e,v}:X_e \longhookrightarrow X_{v}\text{ and }\s_{e,w}:X_e \longhookrightarrow X_{w}. $$
 The triple $\big(T,\left(X_v\right)_{v \in V},\left(X_e\right)_{e \in \E} \big)$ is called a \emph{tree of spaces}.
 We define the \emph{total space} $X$ associated with this tree of spaces as the disjoint union of the $X_v$'s :
 $$ X=\bigsqcup_{v \in V}X_v. $$ 
\end{df}

\begin{center}
 \includegraphics[width=16cm]{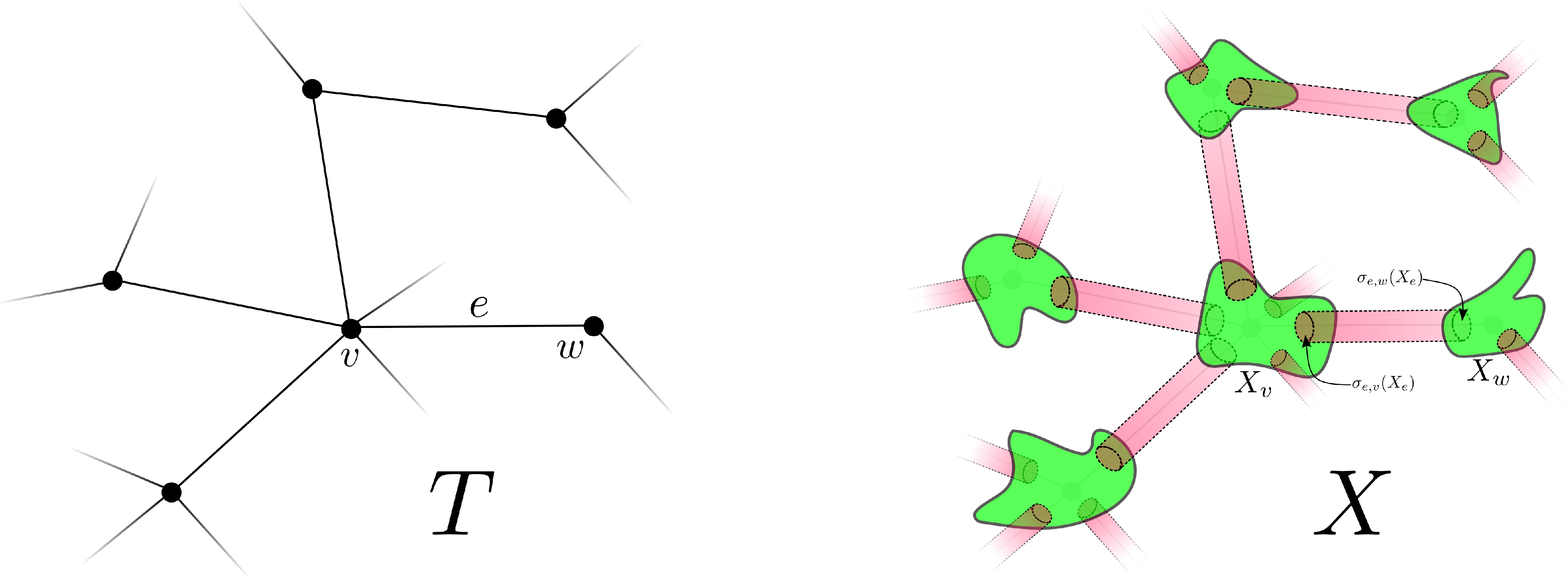} \\
  A tree $T$ and, in green, the total space $X$ of a tree of spaces of base $T$.
 \end{center}
 
 \begin{rmq}
  Some authors consider another definition for the total space $X$ of a tree of spaces (see, for instance, \cite{tu}) which keeps track of the adjacency in the base tree, namely, given an orientation of the edges : 
  $$X=\left(\bigsqcup_{v \in V}X_v \sqcup \bigsqcup_{e \in \E}(X_e \times [0,1])\right)/\sim ,$$
  where the identification $\sim$ is given by, for $e=(v,w) \in \E$, $X_e\times \{0\}\sim \s_{e,v}(X_e)$ and $X_e\times \{1\}\sim \s_{e,w}(X_e).$ This corresponds to the cylinders in the previous figure. \\
  In the present paper, we do not consider this additionnal data in the total space for a matter of simplicity.
 \end{rmq}

\begin{df}
 Let $v \in V$ and $x,y \in X$ with $x \in X_w$ and $y \in X_u$. We say that $X_v$ is between $x$ and $y$ if the vertex $v$ is between $w$ and $u$ in $T$ i.e. $v$ belongs to the vertex path joining $w$ to $u$.
\end{df}

\begin{rmq}
 In the case where the $X_v$'s are metric spaces, the total space $X$ can be naturally endowed with a metric which extends the metric of each $X_v$ and the tree metric (see \cite{guedar}). This metric on $X$ can also be obtained by the labelled partitions metric from the constructions we define in Definition \ref{free_natural} and Definition \ref{free_tree} when each $X_v$ is endowed with a structure of space with labelled partitions (in the case where edge sets are single points).
\end{rmq}
% If the $X_v$'s are metric spaces, the total space $X$ can be endowed with a metric which extends the metric on rfmfzef (see DG).

% citer guentner dadarlat pour le fait de pouvoir mettre une métrique là dessus, mais dans notre cas, on s'en fout, la métrique viendra de l'arbre et des partitions pondérées.

An automorphism of a tree is a bijection $f$ of the vertex set such that $f(v)$ and $f(w)$ are connected by an edge if and only if $v$ and $w$ are connected by an edge. From this notion, we describe what is an automorphism of a tree of spaces :

\begin{df}[Automorphisms of tree of spaces]\label{autom_tree}
Let $\mathcal{X}=\big(T,\left(X_v\right)_{v \in V},\left(X_e\right)_{e \in \E} \big)$ be a tree of spaces and $X$ be its total space. We denote by $\rho : X \rightarrow V$ the natural projection given by $x \in X_v \mapsto v$. \smallskip \\
We say that a bijection $\vp: X \rightarrow X$ is an automorphism of $\mathcal{X}$ if :
\begin{enumerate}
 \item There exists $\widetilde{\vp}:V\rightarrow V$ such that $\widetilde{\vp}$ is an automorphism of $T$ and :
 $$ \widetilde{\vp}\circ \rho = \rho \circ \vp. $$
 \item The restriction $\vp_{|_{\s_{e,v}(X_e)}}$ induces a bijection from $\s_{e,v}(X_e)$ to $\s_{\widetilde{\vp}(e),\widetilde{\vp}(v)}(X_{\widetilde{\vp}(e)})$.
\end{enumerate}

\end{df}

\begin{rmq}\label{free_autom_rmq}
 Let $\vp$ be an automorphism of $\mathcal{X}$.
 \begin{enumerate}
  \item The restriction $\vp_{|_{X_v}}$ is a bijection from $X_v$ to $X_{\widetilde{\vp}(v)}$. \\
  \item The map $\widehat{\vp}_{e,v}:= \s_{\widetilde{\vp}(e),\widetilde{\vp}(v)}^{-1} \circ \vp \circ \s_{e,v}$ is a bijection from $X_e$ to $X_{\widetilde{\vp}(e)}$. \\
  \item The map $\vp^{-1}$ is an automorphism of $\mathcal{X}$ and we have $\widetilde{\vp^{-1}}=\widetilde{\vp}^{-1}$.
 \end{enumerate}

\end{rmq}

Subsequently, we consider a tree of spaces where the edge sets are reduced to single points : \\
Let $\big(T,\left(X_v\right)_{v \in V},\left(\{\bu_e\}\right)_{e \in \E} \big)$ be a tree of spaces and $X$ be its total space.

\begin{center}
 \includegraphics[width=13cm]{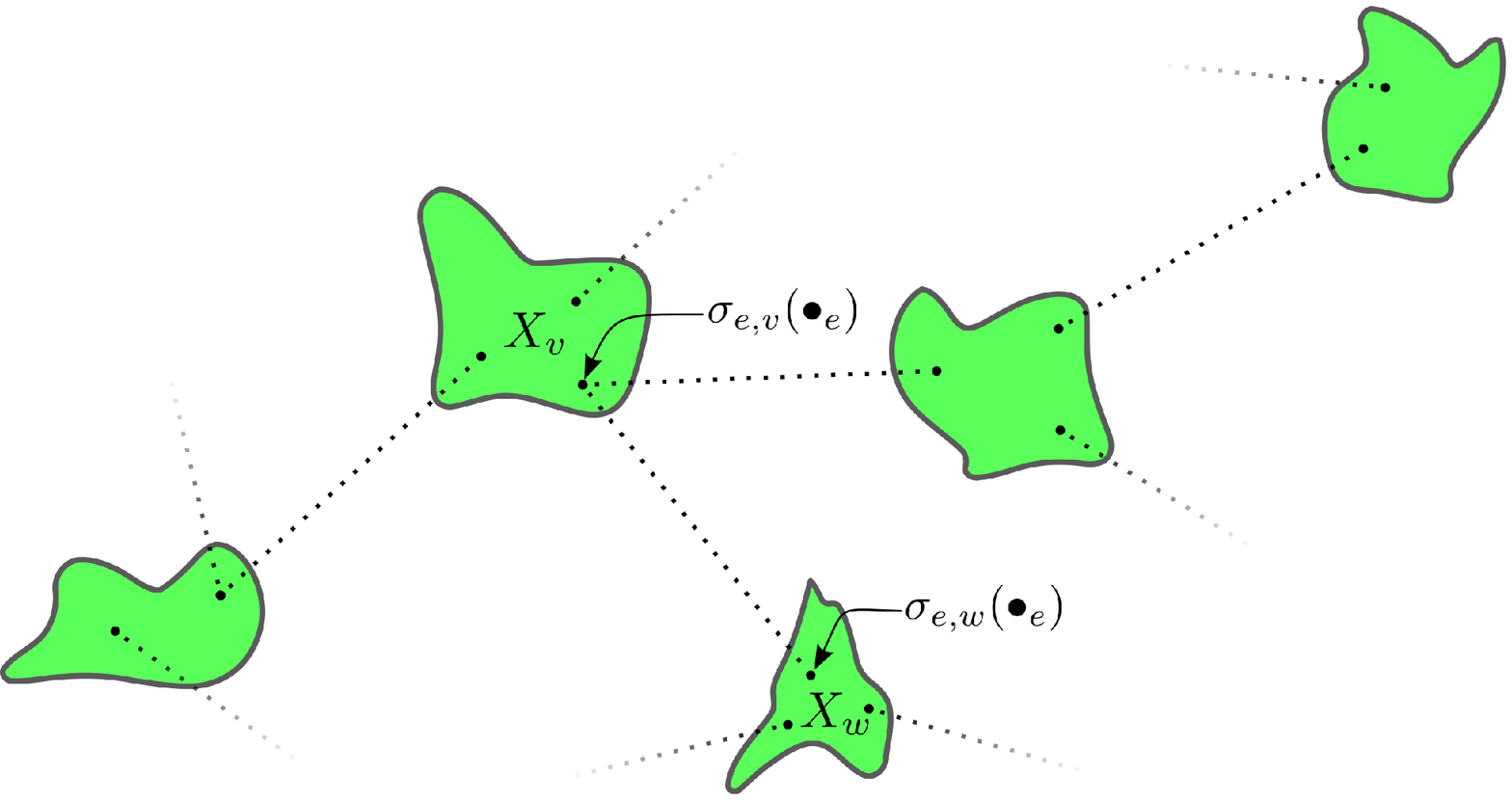} \\
  The total space of a tree of spaces whose edge sets are singletons.
 \end{center}

\begin{df}[Projection on vertex sets]\label{free_proj}
 Let $v \in V$. The map $\pi_v : X \rightarrow X_v$ defined by, for $x \in X_w$ with $w \in V$ :
 $$  \pi_{v}(x)=\left \{
    \begin{array}{ll}
       x &\text{ if }w=v, \\
       \s_{e,v}(\bu_e) &\text{ if } w\neq v, \\
       
    \end{array}
    \right .$$
    where $e$ is the first edge in the edge path in $T$ from $v$ to $w$.
\end{df}

\begin{center}
 \includegraphics[width=11cm]{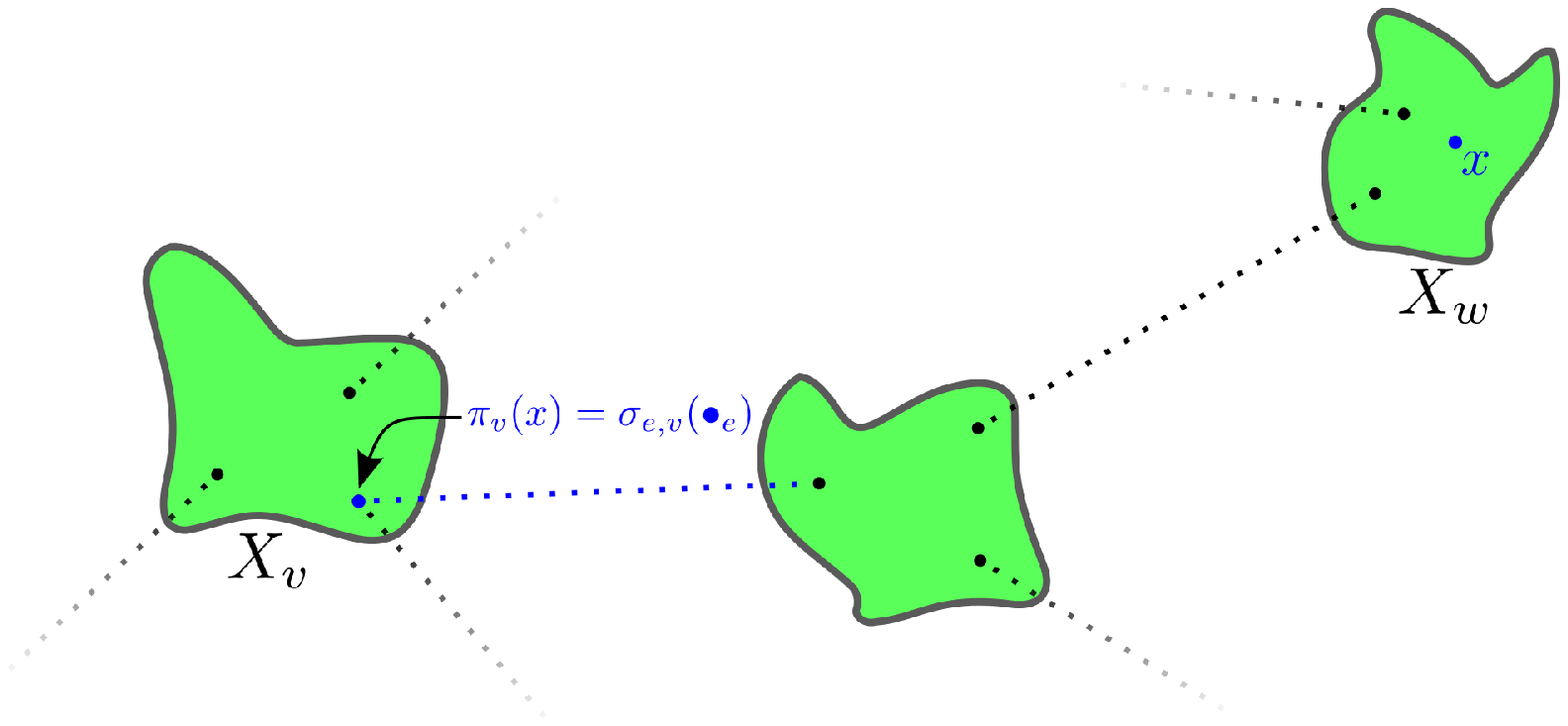} \\
  The projection $\pi_v$ on $X_v$.
 \end{center}

\begin{lem}\label{free_proj_lem}
 Let $x,y \in X$ and $v \in V$. If $\pi_v(x) \neq \pi_v(y)$ then $X_v$ is between $x$ and $y$. In particular, the set $\{v \in V \st \pi_v(x) \neq \pi_v(y) \}$ is finite. 
\end{lem}

\begin{proof}
 Let $x \in X_w$ and $y \in X_u$ with $w,u \in V$. If $v$ is not between $w$ and $u$ in $T$, then $v \neq w$, $v \neq u$ and moreover the edge path joining $v$ to $w$ and the edge path joining $v$ to $u$ in $T$ coincide at least on the first edge $e$. Hence, by definition : 
 $$\pi_{v}(x)=\s_{e,v}(\bu_e)=\pi_{v}(y).$$
\end{proof}

\begin{lem}\label{free_act_proj}
 Let $\vp$ be an automorphism of a tree of spaces $\mathcal{X}=\big(T,\left(X_v\right)_{v \in V},\left(\{\bu_e\}\right)_{e \in \E} \big)$. Then for $v \in V$, we have :
 
 $$ \vp \circ \pi_v=\pi_{\widetilde{\vp}(v)} \circ \vp. $$
\end{lem}

\begin{proof} Let us show that, for all $x$ in the total space $X$ of $\mathcal{X}$, $\vp(\pi_v(x))=\pi_{\widetilde{\vp}(v)}(\vp(x))$. 
 If $x \in X_v$, the identity is clear since $\vp(x)$ belongs to $X_{\widetilde{\vp}(v)}$. Now, assume that $x \in X_w$ with $w \neq v$ and denote by $e$ the first edge in the edge path joining $v$ to $w$ in $T$. As $\widetilde{\vp}$ is an automorphism of $T$, $\widetilde{\vp}(e)$ is the first edge in the edge path joining $\widetilde{\vp}(v)$ to $\widetilde{\vp}(w)$. Hence, since $\vp(x) \in X_{\widetilde{\vp}(w)}$, we have, by Remark \ref{free_autom_rmq}, 2.  :
 $$\pi_{\widetilde{\vp}(v)}(\vp(x))=\s_{\widetilde{\vp}(e),\widetilde{\vp}(v)}(\bu_{\widetilde{\vp}(e)})=\vp(\s_{e,v}(\bu_{e}))=\vp(\pi_v(x)). $$
\end{proof}

\begin{lem}\label{free_act_phibet}
 Let $\vp$ be an automorphism of a tree of spaces $\big(T,\left(X_v\right)_{v \in V},\left(\{\bu_e\}\right)_{e \in \E} \big)$. For $x,y \in X$, we have :
 
 $$ \{v \in V \st \pi_v(\vp(x)) \neq \pi_v(\vp(y)) \}=\widetilde{\vp}\left( \{v \in V \st \pi_v(x) \neq \pi_v(y) \} \right).$$
\end{lem}

\begin{proof}
 Let $x,y \in X$. By Lemma \ref{free_act_proj}, we have, for $v \in V$ : \smallskip
 
 \begin{center}
 \begin{tabular}{rl}
  $\pi_v(\vp(x)) \neq \pi_v(\vp(y))$&$\Leftrightarrow \vp (\pi_{\widetilde{\vp}^{-1}(v)}(x)) \neq \vp(\pi_{\widetilde{\vp}^{-1}(v)}(y))$ \smallskip \\
  $\;$&$\Leftrightarrow \pi_{\widetilde{\vp}^{-1}(v)}(x) \neq \pi_{\widetilde{\vp}^{-1}(v)}(y), \;$ since $\vp$ is a bijection. \\
 \end{tabular}
 \end{center}
 Thus, 
 \begin{center}
 \begin{tabular}{rl}
  $\{v \in V \st \pi_v(\vp(x)) \neq \pi_v(\vp(y)) \}$&$=\{v \in V \st \pi_{\widetilde{\vp}^{-1}(v)}(x) \neq \pi_{\widetilde{\vp}^{-1}(v)}(y) \}$ \smallskip \\
  $\;$&$=\{\widetilde{\vp}(v) \in V \st \pi_v(\vp(x)) \neq \pi_v(\vp(y)) \}.$ \\
 \end{tabular}
 \end{center}

\end{proof}

\begin{exemple}\label{free_example}
 We consider an amalgamated free product $\G=G*_C H$, together with the natural action on its Bass-Serre tree $T=(V,\E)$ where :
 $$V= \G/G \sqcup \G/H \text{ and } \E=\G/C,$$
 and the endpoints maps are given by the inclusions of $C$ left-cosets into $G$ and $H$ left-cosets. \\
 
 For our purpose, we construct the following tree of space of base $T$ : \medskip \\
 - For $v=\g G \in V$, we consider $X_v = \g G/C=\{\g g C \st g\in G\}$ and for $v=\g H$, we set $X_v = \g H/C$. \smallskip \\
 - For $e=\g C \in \E$, we consider the singleton $X_e=\{\g C\}$. The structural maps $\s_{\g C,\g G}$ and $\s_{\g C,\g H}$ are the trivial maps $\g C \mapsto \g C \in \g G/C$ and $\g C \mapsto \g C \in \g H/C$. \smallskip \\
 
 These datas give rise to a tree of spaces $\mathcal{X}=(T,\{X_v\},\{X_e\})$ on which $\G$ acts by automorphisms of tree of spaces. \\
 In fact, every $\g \in \G$ defines a bijection of the total space $X$ by :
 $$\g'g C \in X_{\g'G} \mapsto \g\g'g C \in X_{\g\g'G}$$ and,
 $$\g'h C \in X_{\g'H} \mapsto \g\g'h C \in X_{\g\g'H}.$$
 Moreover, the map $\widetilde{\g}:V \rightarrow V$ is exactly the map $v \mapsto \g v$ given by the action of $\G$ on $T$ and we have : 
 $$\s_{\g\g' C,\g\g' G}(\bu_{\g\g' C})= \g\g' C = \g \s_{\g'C,\g'G}(\bu_{\g'C}).$$
 Thus $\s_{\g\g' C,\g\g' G}^{-1}\circ \g \circ \s_{\g'C,\g'G}$ is a bijection and similarly, $\s_{\g\g' C,\g\g' H}^{-1}\circ \g \circ \s_{\g'C,\g'H}$ is a bijection, for all $\g' \in \G$.
\end{exemple}

\begin{df}\label{free_example_df}
 Let $\G=G*_C H$ be an amalgamated and $T$ be its Bass-Serre tree. We call \emph{tree of $C$-cosets spaces associated with $\G$}, the tree of spaces $(T,\{X_{v}\},\{X_e\})$ defined in Example \ref{free_example}.
\end{df}

\begin{center}
 \includegraphics[width=11cm]{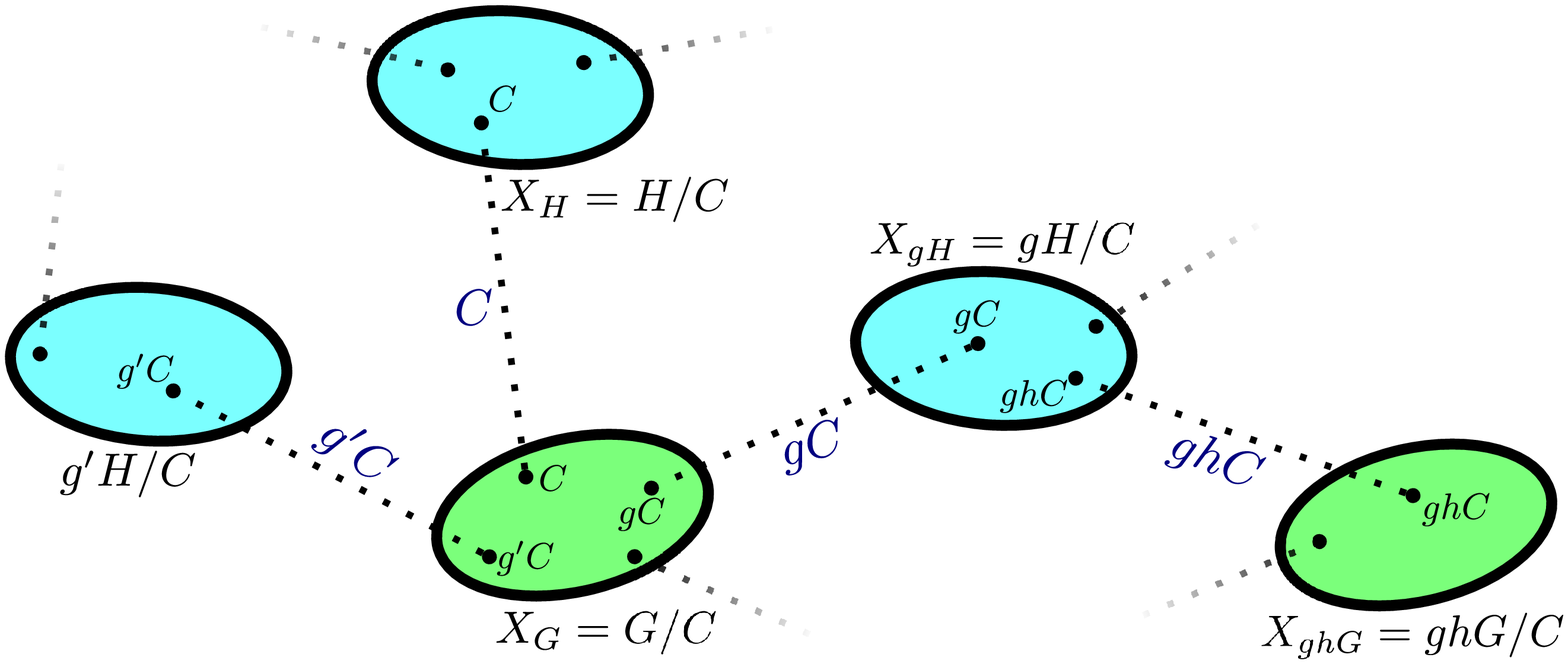} \\
  The tree of $C$-cosets spaces associated with $G*_C H$.
 \end{center}

 %def lab part tree of spaces induced by vertex lab part
 
 \subsubsection{Labelled partitions induced by the vertex sets}
 
 \begin{df}
 Let $\big(T,\left(X_v\right)_{v \in V},\left(\{\bu_e\}\right)_{e \in \E} \big)$ be a tree of spaces and $X$ be its total space. Assume that each vertex set $X_v$ is endowed with a structure of space with labelled partitions $(X_v,\P_v,F_v(\P_v))$.

 Let $v \in V$. We set, for $p_v \in \P_v$, the following labelling function on $X$ :
 $$ p_v^{\oplus_v}=p_v\circ \pi_{v}, $$
 and we denote $ \P_v^{\oplus_v}=\{p_v^{\oplus_v}\st p_v \in \P_v \}$. \\
 The set :
 $$ \P_X=\bigcup_{v \in V} \P_v^{\oplus_v} $$
 is called the \emph{family of labelling functions induced by the vertex sets}.
 \end{df}
 
  \begin{center}
   \includegraphics[width=14cm]{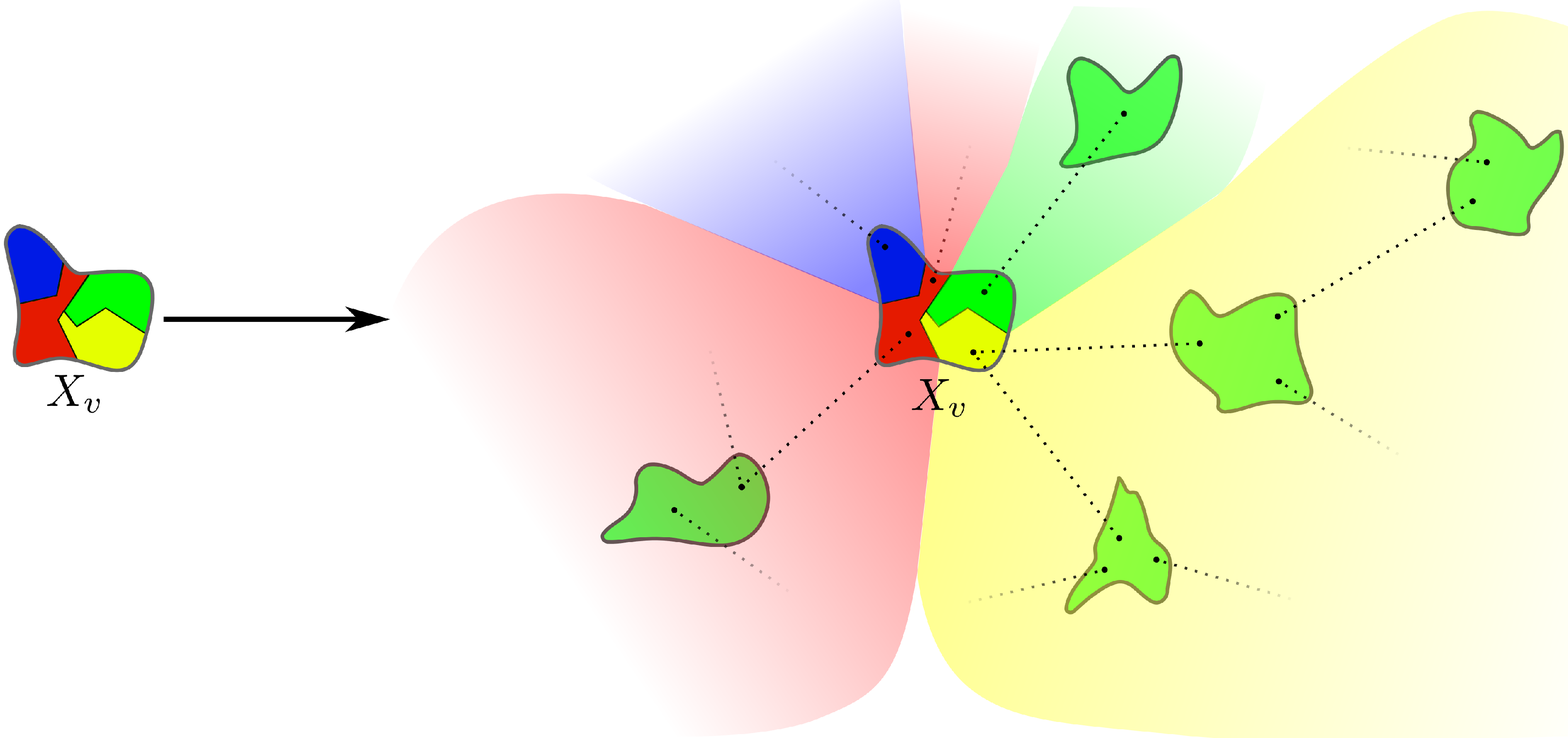} \\
   Partition of $X$ induced by a partition of $X_v$ via the projection $\pi_v$.
  \end{center}

\begin{df}
Let $\big(T,\left(X_v\right)_{v \in V},\left(\{\bu_e\}\right)_{e \in \E} \big)$ be a tree of spaces and $X$ be its total space. Assume that each vertex set $X_v$ is endowed with a structure of space with labelled partitions $(X_v,\P_v,F_v(\P_v))$. \\
 Let $v\in V$. For $\k \in F_v(\P_v)$, we denote $\k^{\oplus_v}: \P_{X} \rightarrow \K$ the function :
\begin{center}\begin{equation*}
   \k^{\oplus_v}(p)=\begin{cases} \k(p_v)\text{ if }p=p_v^{\oplus_v} \in \P_v^{\oplus_v} \\
                            0\;\;\;\;\:\:\text{ otherwise.}
                           \end{cases}
   \end{equation*}
  \end{center}
  Let $q \geq 1$. We set : 
  $$ E_q(\P_{X}):=\left\lbrace \sum_{v \in V} \k_{v}^{\oplus_v} \st \k_{v} \in F_v(\P_v)\text{ with }\k_{v} = 0  \text{ for all but finitely many vertices }v \right\rbrace,$$ endowed with the norm $\|.\|_{q}$ defined by, for $\k=\sum_{v} \k_{v}^{\oplus_v}$:
  $$ \|\k\|_{q}:=\left(\sum_{v}\|\k_{v}\|_{_{F_v(\P_v)}}^q\right)^{\frac{1}{q}}.$$
  The Banach space $F_q(\P_{X}):=\overline{E_q(\P_{X})}^{\|.\|_{q}}$ is called the \emph{$q$-space of functions on $\P_{X}$ of $X$}.
\end{df}

\begin{prop}[Labelled partitions structure on $X$ induced by the vertex sets]\label{const_labfree_prop} $\;$ \\
Let $\big(T,\left(X_v\right)_{v \in V},\left(\{\bu_e\}\right)_{e \in \E} \big)$ be a tree of spaces and $X$ be its total space. Assume that each vertex set $X_v$ is endowed with a structure of space with labelled partitions $(X_v,\P_v,F_v(\P_v))$. Consider $X$ together with its family $\P_{X}$ of labelling functions induced by the vertex sets.  \medskip \\
 Let $q \geq 1$ and $F_q(\P_{X})$ be the $q$-space of functions on $\P_{X}$ of $X$. Then, the triple $(X,\P_{X},F_q(\P_{X}))$ is a space with labelled partitions.
 Moreover, we have, for $x,y \in X$ :
 $$ \|c(x,y)\|_{_{F_q(\P_{X})}}^q=\sum_{v \in V}\|c_v(\pi_v(x),\pi_v(y))\|_{_{F_v(\P_v)}}^q $$
\end{prop}

\begin{proof}
We denote by $c_v$ the separation map of $X_v$ associated with $\P_v$ and by $c_{X}$ the separation map associated with $\P_{X}$. \smallskip \\
Let $x,y \in X$ and $v \in V$. For $p_v^{\oplus_v} \in \P_v^{\oplus_v}$, we have: 
$$c_{X}(x,y)(p_v^{\oplus_v})=p_v(\pi_{v}(x))-p_v(\pi_{v}(y))=c_v(\pi_{v}(x),\pi_{v}(y))(p_v).$$
It follows that $c_{X}(x,y)=\sum_{v} c_v(\pi_{v}(x),\pi_{v}(y))^{\oplus_v}$ which is a finite sum since $\pi_{v}(x)=\pi_{v}(y)$ for all but finitely many $v$'s by Lemma \ref{free_proj_lem}. Thus, $c_{X}(x,y)$ belongs to $F_q(\P_{X})$ and hence, $(X,\P_{X},F_q(\P_{X}))$ is a space with labelled partitions. \\
\end{proof}

\begin{df}\label{free_natural}
  Let $\big(T,\left(X_v\right)_{v \in V},\left(\{\bu_e\}\right)_{e \in \E} \big)$ be a tree of spaces and $X$ be its total space. Assume that each vertex set $X_v$ is endowed with a structure of space with labelled partitions $(X_v,\P_v,F_v(\P_v))$. Consider $X$ together with its family $\P_{X}$ of labelling functions induced by the vertex sets and let $F_q(\P_{X})$ be the $q$-space of functions on $\P_{X}$ of $X$. \medskip \\
  The triple $(X,\P_{X},F_q(\P_{X}))$ is called the \emph{space with labelled partitions on $X$ induced by the vertex sets}.
\end{df}

\begin{prop}\label{free_autom_part}
 Let $\mathcal{X}=\big(T,\left(X_v\right)_{v \in V},\left(\{\bu_e\}\right)_{e \in \E} \big)$ be a tree of spaces and $X$ be its total space. Assume that each vertex set $X_v$ is endowed with a structure of space with labelled partitions $(X_v,\P_v,F_v(\P_v))$ and let $(X,\P_{X},F_q(\P_{X}))$ be the space with labelled partitions on $X$ induced by the vertex sets. \smallskip \\
 Let $\vp$ be an automorphism of $\mathcal{X}$. If, for all $v \in V$, the map $\vp_{|_{X_v}}:X_v \rightarrow X_{\widetilde{\vp}(v)}$ is a homomorphism of spaces with labelled partitions, then $\phi$ is an automorphism of space with labelled partitions of $(X,\P_{X},F_q(\P_{X}))$. 
\end{prop}

\begin{proof}
 Let $p=p_v^{\oplus_v} \in \P_X$ with $v \in V$ and $p_v \in \P_v$. Then we have, by Lemma \ref{free_act_proj} :
 $$ \phi_{\vp}(p)=p_v \circ \pi_v \circ \vp = p_v \circ \vp_{|_{X_{\widetilde{\vp}^{-1}(v)}}} \circ \pi_{\widetilde{\vp}^{-1}(v)}=(p_v \circ \vp_{|_{X_{\widetilde{\vp}^{-1}(v)}}})^{\oplus_{\widetilde{\vp}^{-1}(v)}} \in \P_X,$$
 since $\vp_{|_{X_{\widetilde{\vp}^{-1}(v)}}}:X_{\widetilde{\vp}^{-1}(v)} \rightarrow X_v$ is a homomorphism of space with labelled partitions. \\
 Now, let $\k = \sum_{v \in V} \k_{v}^{\oplus_v} \in E_q(\P_X)$ where $\k_v \in F_v(\P_v)$ for all $v \in V$. We have, for $p=p_v^{\oplus_v} \in \P_X$ :
 \begin{center}
  \begin{tabular}{rl}
   $\k \circ \phi_{\vp}(p) $&$=\k((p_v \circ \vp_{|_{X_{\widetilde{\vp}^{-1}(v)}}})^{\oplus_{\widetilde{\vp}^{-1}(v)}})$, \medskip \\
   $\;$&$=\k_{\widetilde{\vp}^{-1}(v)}(p_v \circ \vp_{|_{X_{\widetilde{\vp}^{-1}(v)}}}), $ \medskip \\
   $\k \circ \phi_{\vp}(p) $&$=(\k_{\widetilde{\vp}^{-1}(v)}\circ \phi_{\vp_{|_{X_{\widetilde{\vp}^{-1}(v)}}}})^{\oplus_v}(p).$
  \end{tabular}
 \end{center}
 Thus, 
 $$ \k \circ \phi_{\vp} = \sum_{v \in V} (\k_{\widetilde{\vp}^{-1}(v)}\circ \phi_{\vp_{|_{X_{\widetilde{\vp}^{-1}(v)}}}})^{\oplus_v}. $$
 Hence, by making the substitution $v=\widetilde{\vp}^{-1}(v)$, we obtain :
 $$ \k \circ \phi_{\vp} = \sum_{v \in V} (\k_{v}\circ \phi_{\vp_{|_{X_{v}}}})^{\oplus_{\widetilde{\vp}(v)}} \in E_q(\P_X), $$
 since $\k_{v}\circ \phi_{\vp_{|_{X_{v}}}}$ belongs to $F_{\widetilde{\vp}(v)}(\P_{\widetilde{\vp}(v)})$ and $\k_{v}=0$ for all but finitely many $v$'s. \smallskip \\
 By completeness of $F_q(\P_X)$, we have, for all $\k \in F_q(\P_X)$, $\k \circ \phi_{\vp} \in F_q(\P_X)$. \medskip \\
 Moreover, since, for all $v \in V $, $\|\k_v \circ \phi_{\vp_{|_{X_{v}}}}\|_{_{F_{\widetilde{\vp}(v)}(\P_{\widetilde{\vp}(v)})}}=\|\k_v\|_{_{F_{v}(\P_{v})}}$, it follows that : \medskip 
 \begin{center}
  \begin{tabular}{rl}
   $\|\k \circ \phi_{\vp}\|_{_{q}}^q $&$=\sum_{v \in V} \|\k_{v}\circ \phi_{\vp_{|_{X_v}}}\|_{_{F_{\widetilde{\vp}(v)}(\P_{\widetilde{\vp}(v)})}}^q$, \medskip \\
   $\;$&$=\sum_{v \in V}\|\k_v\|_{_{F_{v}(\P_{v})}}^q, $ \medskip \\
   $\|\k \circ \phi_{\vp}\|_{_{q}}^q  $&$=\|\k\|_{_{q}}^q.$
  \end{tabular}
 \end{center} \smallskip 
 Then, $\vp$ is an automorphism of $(X,\P_{X},F_q(\P_{X}))$.
\end{proof}

 %\subsubsection{labelled partitions induced by the edge sets}

  %def lab part induced by the tree
 
 \subsubsection{Labelled partitions induced by the tree structure}
 
 We detail here the space with labelled partitions induced by the natural wall structure on the set of vertices of a tree $T$, and we consider the pullback of this structure on a tree of spaces of base $T$ via the projection $\rho : X \rightarrow V$ namely, for $x \in X_v \subset X$, $\rho(x)=v$. Here we denote by $[v,w]_{\E}$ the set of edges in the edge path between two vertices $v$ and $w$.  \medskip \\
 
 Let $q \geq 1$ and let $T=(V,\E)$ be a tree. Let $e$ be an edge with endpoints $v$ and $w$ in $V$. ``Removing'' this edge from $T$ gives rise to two complementary connected components of vertices, namely $h_{e,v}=\{u \in V \st d_T(v,u)< d_T(w,u)\}$ and $h_{e,w}=\{u \in V \st d_T(w,u)< d_T(v,u)\}$.
 
 \begin{center}
 \includegraphics[width=10cm]{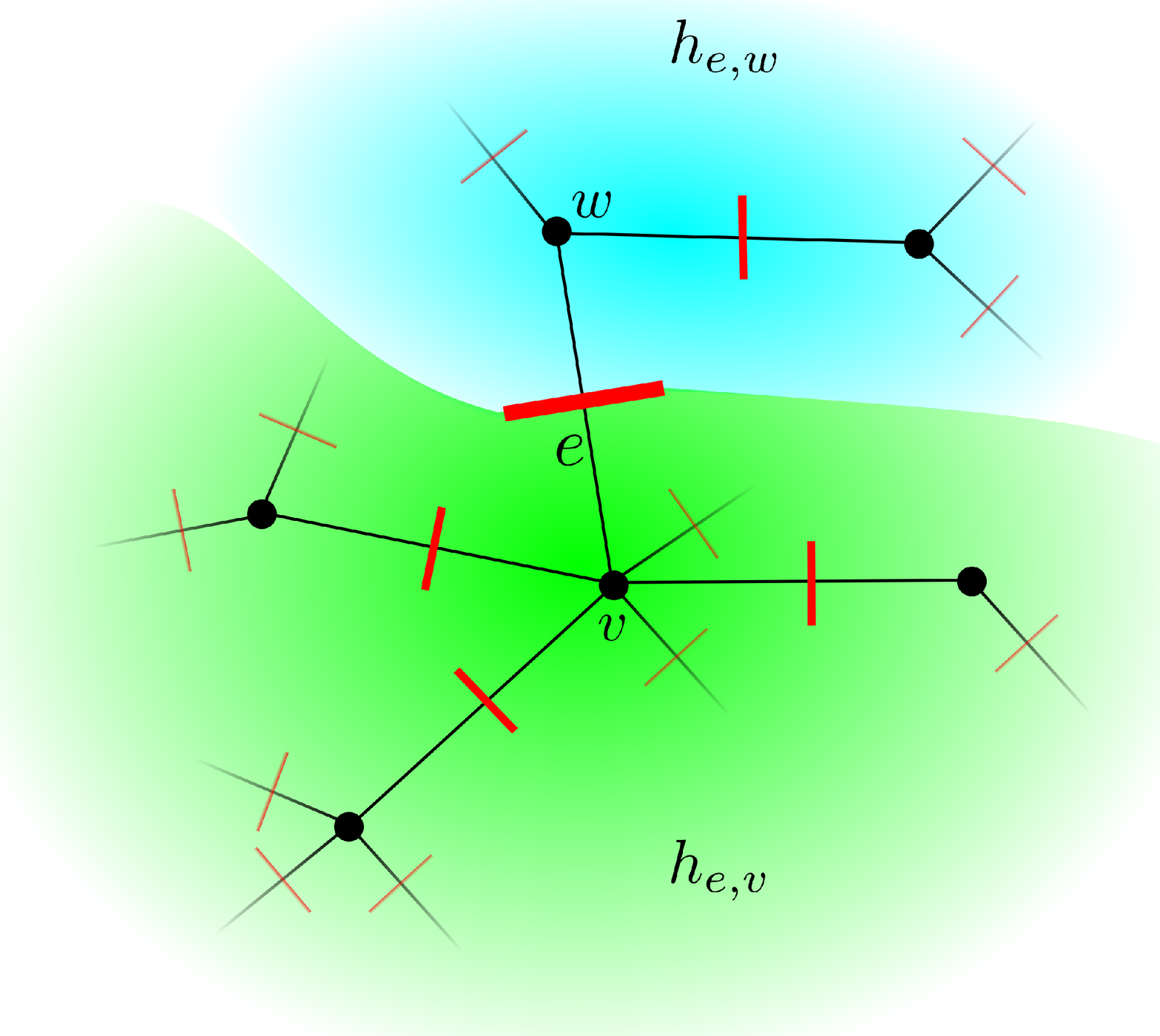} \\
  Wall structure on a tree induced by the edge set.
 \end{center}
 
 We then consider the following family of labelling functions on $V$ : 
 $$ \P = \{ 2^{-\frac{1}{q}}\;\cara_{h_{e,v}} \st e \in \E\text{ and } v \text{ is a endpoint of }e\},$$
 where $\cara_{h_{e,v}}$ is the characteristic function of the set $h_{e,v}$. \smallskip \\
 Notice that, for $v,w \in V$ : 
 \begin{center}\begin{equation*}
   |\cara_{h_{e,u}}(v)-\cara_{h_{e,u}}(w)|=\begin{cases} 1\text{ if }e \in [v,w]_{\E} \\
                            0\text{ otherwise.}
                           \end{cases}
   \end{equation*}
  \end{center}
 Hence, we have, for $v,w \in V$ :
 
 \begin{center}
  \begin{tabular}{rl}
  $\|c(v,w)\|_{\ell^q(\P)}^q$&$=\sum_{p \in \P}|p(v)-p(w)|^q,$ \smallskip \\
  $ \;$ &$=\frac{1}{2}\sum_{e \in \E} \sum_{u \in e} |\cara_{h_{e,u}}(v)-\cara_{h_{e,u}}(w)|^q,$ \medskip \\
  $ \;$ &$=\frac{1}{2} \sum_{e \in [v,w]_{\E}} \;2\;, $ \medskip \\
  $\|c(v,w)\|_{\ell^q(\P)}^q$ &$=\#[v,w]_{\E}=d_T(v,w)$ \smallskip \\
 \end{tabular}
 \end{center}
 
 Hence, $(V,\P,\ell^q(\P))$ is a space with labelled partitions. \medskip \\
 
 We can now consider the pullback (see Definition \ref{pullbackk}) of this space with labelled partitions via the projection $\rho: X \rightarrow V$ :
 
 \begin{df}\label{free_tree}
  Let $\mathcal{X}=\big(T,\left(X_v\right)_{v \in V},\left(X_e\right)_{e \in \E} \big)$ be a tree of spaces and $X$ be its total space.
  The space with labelled partitions on $X$ defined as the pullback of $(V,\P,\ell^q(\P))$ via $\rho :X \rightarrow V$ is called the \emph{structure of labelled partitions of $\mathcal{X}$ induced by the tree structure}.
 \end{df}
 
 \begin{prop}\label{free_tree_prop}
  Let $\mathcal{X}=\big(T,\left(X_v\right)_{v \in V},\left(X_e\right)_{e \in \E} \big)$ be a tree of spaces, $X$ be its total space and $\vp: X \rightarrow X$ be an automorphism of $\mathcal{X}$. \\
  Then $\vp$ is an automorphism of the space with labelled partitions of $\mathcal{X}$ induced by the tree structure.
 \end{prop}
 
 \begin{proof}
  Let $\vp$ be an automorphism of $\mathcal{X}$. Since $\widetilde{\vp}$ is an automorphism of $T$, one can easily show that $\widetilde{\vp}$ is an automorphism of space with labelled partitions of $(V,\P,\ell^q(\P))$. Hence, as $\vp \circ \phi = \phi \circ \widetilde{\vp}$, by Lemma \ref{pullbackpart}, $\phi$ is an automorphism the space with labelled partitions of $\mathcal{X}$ induced by the tree structure.
 \end{proof}

 %Proposition and formulas
 
 \subsection{Amalgamated free product and property $PL^p$}\label{free_amalg} 

 Let $\G=G*_C H$ be an amalgamated free product and let us consider the tree of $C$-cosets spaces $\mathcal{X}=\big(T,\left(X_v\right)_{v \in V},\left(\{\bu_e\}\right)_{e \in \E} \big)$ on which $\G$ acts by automorphisms (see Definition \ref{free_example_df}. \\
 Recall that $X_v=\g G/C$ if $v=\g G$, $X_v=\g H/C$ if $v=\g H$ and $\bu_{\g C} = \g C$. \\ 
 
 Let us consider systems of representatives $\rep(G/C)$ and $\rep(H/C)$ of $G/C$ and $H/C$ respectively, each containing the unit of the group as the representative of the class $C$. Every element in $\G$ can be expressed as a reduced word in terms of these systems of representatives (see, for instance, \cite{scowal}) :
 
  \begin{dfpr}[reduced word]\label{normalform}
  Let $G,H$ be groups and $C$ be a common subgroup. An element $\g$ of $G*_C H$ can be uniquely written in the following way :
  $$ \g = g_1h_1...g_nh_n c, $$
  where :
  \begin{itemize}
  \item[-] for $i=1,...,n$, $g_i \in \rep(G/C)$ and $h_i \in \rep(H/C)$;
  \item[-] for $i > 1$, $g_i\neq e_G $ and for $i < n$, $h_i \neq e_H$;
  \item[-] $c \in C$.
  \end{itemize}
  Such an expression of $\g$ is called a \emph{reduced word} (relatively to $\rep(G/C)$ and $\rep(H/C)$). \medskip \\
  Let us denote : \smallskip \\ 
 \indent $ R_{\G/G}=\{ g_1h_1...g_nh_n \st g_1h_1...g_nh_n \text{ is a reduced word with} h_n \neq e_H \},$ and \smallskip \\
 \indent $ R_{\G/H}=\{ g_1h_1...g_n \st g_1h_1...g_n \text{ is a reduced word} \}.$ (notice that $g_n\neq e_G$ by definition).
 Then $R_{\G/G}$ is a system of representatives of $\G/G$ in $\G$ and $R_{\G/H}$ is a system of representatives of $\G/H$ in $\G$.
  
 \end{dfpr}
 
 The maps we define below will allow us to endow the vertex sets of the tree of $C$-cosets spaces of $\G$ with the pullback structure of space with labelled partitions coming from $G/C$ and $H/C$.
 
 \begin{nt}\label{pullmaps}
 
  Let $\g \in R_{\G/G}$ and $\g' \in R_{\G/H}$. We set : \smallskip
 \begin{itemize}
  \item[-] $f_{\g G}: X_{\g G}\rightarrow G/C$ such that $\g g C \mapsto g C$ and, \\
  \item[-] $f_{\g' H}: X_{\g' H}\rightarrow H/C$ such that $\g' h C \mapsto h C$.
 \end{itemize}
 \end{nt}
 
 These maps satifies the following equivariance formulas :

 \begin{lem}\label{free_formula_coset}
  Let $\g \in \G$. We have : \\
  - Let $\g_1, \g_2 \in R_{\G/G}$. Assume there exists $g \in G$ such that $\g \g_1 = \g_2 g$. Then, for all $x \in X_{\g_1 G}$ :
  $$ f_{\g_2 G}(\g x)= g f_{\g_1 G}(x).$$
  - Let $\g_1, \g_2 \in R_{\G/H}$. Assume there exists $h \in H$ such that $\g \g_1 = \g_2 h$. Then, for all $x \in X_{\g_1 H}$ :
  $$ f_{\g_2 H}(\g x)= h f_{\g_1 H}(x).$$
 \end{lem}
 
 \begin{proof}
  Let $\g \in \G$ and $\g_1,\g_2 \in R_{\G/G}$ such that there exists $g \in G$ such that $\g \g_1 = \g_2 g$. For $x \in X_{\g_1 G}$, there exists $g_x \in G$ such that $x= \g_1g_x C$. Then we have :
  $$f_{\g_2 G}(\g x) = f_{\g_2 G}(\g_2 g g_xC)=g g_x C=g f_{\g_1 G}(x).$$
  A similar argument holds for the second statement.
 \end{proof}

 \begin{theo}\label{free_act_theo}
  Let $G,H$ be groups and $C$ be a common subgroup.
  Assume that $G$ acts by automorphisms on $(G/C,\P,F(\P))$ and $H$ acts by automorphisms on $(H/C,\P',F'(\P'))$. \\
  Let $\mathcal{X}$ be the tree of $C$-cosets spaces associated with $\G=G*_C H$ and $X$ be the total space of $\mathcal{X}$.
  Then $\G$ acts by automorphisms on $(X,\P_X,F_X(\P_X))$ such that $F_q(\P_X)$ is isometrically isomorphic to a closed subspace of $\rexp{q}{\bigoplus} F(\P)\oplus \rexp{q}{\bigoplus} F'(\P') \oplus \ell^q$. \\
  Moreover, considering $C \in X_G \subset X$, we have, for $\g = g_1h_1...g_nh_n c \in G*_C H$ :
  $$ \|c_X(\g C,C)\|^q=\sum_{k=1}^n (\|c_G(g_k C,C)\|^q_{_{F(\P)}}+\|c_H(h_kC,C)\|^q_{_{F'(\P')}})+d_T(\g G,G)^q. $$ \smallskip
  In particular, if $G \curvearrowright (G/C,\P,F(\P))$ and $H \curvearrowright (H/C,\P',F'(\P'))$ are proper, then $\G \curvearrowright (X,\P_X,F_X(\P_X))$ is proper.
 \end{theo}
 
 \begin{proof}
 $\mathcal{X}=\big(T,\left(X_v\right)_{v \in V},\left(\{\bu_e\}\right)_{e \in \E} \big)$ tree of $C$-cosets spaces associated with $\G=G*_C H$ and let $R_{\G/G}$ and $R_{\G/H}$ be the systems of representative of $\G/G$ and $\G/H$ respectively defined in Definition-Proposition \ref{normalform}. \\
 Notice that $V=\{\g G \st \g \in R_{\G/G}\} \sqcup \{\g H \st \g \in R_{\G/H}\}$. \smallskip \\
 
  For $ \g G, \g' H \in V$ with $\g \in R_{\G/G}$ and $\g' \in R_{\G/H}$, we endow the vertex sets $X_{\g G}$ and $X_{\g H}$ with pullback structures of spaces with labelled partitions $(X_{\g G},\P_{\g G},F_{\g G}(\P_{\g G})$ and $(X_{\g H},\P_{\g H},F_{\g H}(\P_{\g H})$ given, respectively, by the maps introduced in Notation \ref{pullmaps} :\\
  - $f_{\g G}: X_{\g G}\rightarrow G/C$ such that $\g g C \mapsto g C$  and, \\
  - $f_{\g' H}: X_{\g' H}\rightarrow H/C$ such that $\g' H C \mapsto h C$. \smallskip \\
  
  Now, from these structures, we consider the space with labelled partitions $(X,\P_{V},F_q(\P_{V}))$ induced by the vertex sets given by Definition \ref{free_natural} and we denote by $c_V$ the associated separation map. \\
  We must prove that $\G$ acts by automorphisms of space with labelled partitions on $(X,\P_{V},F_q(\P_{V}))$. We showed in Example \ref{free_example} that $\G$ acts by automorphisms of tree of spaces on $\mathcal{X}$. Hence, by Lemma \ref{autom_tree}, it sufficies to show that, for all $\g \in \G$, the map $\g_{|_{X_v}}:X_v \rightarrow X_{\g v}$ is a homomorphisms of space with labelled partitions, for every $v \in V$: \smallskip \\
  Let $\g \in \G$. Let $v = \g_1 G \in V$ with $\g_1 \in R_{\G/G}$ and let $\g_2 \in R_{\G/G}$ such that $\g \g_1 = \g_2 g$ for some $g \in G$ i.e. $\g_2$ is the representative in $R_{\G/G}$ of the coset $\g \g_1 G$. Notice that $\g v = \g_2 G$. \\
  A labelling function of $\P_{\g v}$ is of the form $p \circ f_{\g_2 G}$ for some $p \in \P$, and we have, by Lemma \ref{free_formula_coset}, for all $x \in X_v$,$ p(f_{\g_2 G}(\g x))=p(g f_{\g_1 G}(x)).$ \smallskip \\
  Let us set the following maps : \smallskip \\
  - $\phi_{f_{\g_1 G}}: \P \rightarrow \P_v$ such that $p \mapsto p \circ f_{\g_1 G}$; \smallskip \\
  - $\phi_{f_{\g_2 G}}: \P \rightarrow \P_{\g v}$ such that $p \mapsto p \circ f_{\g_2 G}$;\smallskip \\
  - $\phi_{g}: \P \rightarrow \P$ such that $p \mapsto \{x \in G/C \mapsto p(gx)\}$ and\smallskip \\
  - $\phi_{\g}:\P_{\g v} \rightarrow \{\K\text{-valued functions on }X_v\}$ such that $p_{\g v} \mapsto p_{\g v} \circ \g_{|_{X_v}}$;\medskip \\
  Thus, by the previous equality, we have :
  $$ \phi_{\g}(p\circ f_{\g_2 G})=\phi_g(p) \circ f_{\g_1 G} \in \P_{X_v}, \; (\ast) $$
  since $\phi_g(p)$ belongs to $\P$. \medskip \\
  Now, let $\k \in F_v(\P_v)$. By using the definitions of pullback structures, we have :
  \begin{center}
   \begin{tabular}{rl}
    $\|\k \circ \phi_{\g}\|_{_{F_{\g v}(\P_{\g v})}}$&$= \|\k \circ \phi_{\g}\circ \phi_{f_{\g_2 G}}\|_{_{F(\P)}}$, \smallskip \\
    $\;$&$= \|\k \circ \phi_{f_{\g_1 G}} \circ \phi_{g}\|_{_{F(\P)}}$, by $(\ast)$, \smallskip \\
    $\;$&$= \|\k \circ \phi_{f_{\g_1 G}}\|_{_{F(\P)}}$, by $(\ast)$, \smallskip \\
    $\|\k \circ \phi_{\g}\|_{_{F_{\g v}(\P_{\g v})}}$&$= \|\k\|_{_{F_v(\P_v)}}$, \medskip \\
   \end{tabular}
  \end{center}
  Hence, $\g_{|_{X_v}}:X_v \rightarrow X_{\g v}$ is a homomorphism of space with labelled partitions and similar argument holds for vertices $v$ of the form $v = \g_1 H$ with $\g_1 \in R_{\G/H}$. As said before, it follows that $\G$ acts by automorphisms on $(X,\P_{V},F_q(\P_{V}))$. \bigskip \\
  
  Let $\g = g_1h_1...g_nh_n c \in G*_C H$ be a reduced word and consider the element $C \in X_G$. By Proposition \ref{const_labfree_prop}, we have :
  $$ \|c_V(\g C, C)\|_{_{F_q(\P_{V})}}^q=\sum_{v \in V}\|c_v(\pi_v(\g C),\pi_v(C))\|_{_{F_v(\P_v)}}^q. $$
  By Lemma \ref{free_proj_lem}, this sum is a finite sum over a subset of $\{v \in V \st v \text{ is between }G \text{ and }\g G\}$.
  But the vertices between $G$ and $\g G$ are the following : 
  $$ G,\;\;g_1H,\;\;\cdots\;,\; g_1h_1...h_{k-1}G,\;\;g_1h_1...h_{k-1}g_k H,\;\;g_1h_1...g_kh_k G,\;\;\cdots\;,\;g_1h_1...g_n H,\;\;\g G ,$$
  and notice that :
  $$ g_1h_1...h_{k-1}g_k C \text{ is the edge between }g_1h_1...h_{k-1}G\text{ and }g_1h_1...h_{k-1}g_k H,$$
  and 
  $$ g_1h_1...g_kh_k C \text{ is the edge between }g_1h_1...h_{k-1}g_kH\text{ and }g_1h_1...g_kh_k G.$$
  It follows that, for $v=g_1h_1...h_{k-1}g_kH$, 
  $$ \pi_v(C)=g_1h_1...h_{k-1}g_k C \text{ and } \pi_v(\g C)=g_1h_1...g_kh_k C.$$
  Then, by denoting $\g_k=g_1h_1...h_{k-1}g_k$, we have :
  \begin{center}
   \begin{tabular}{rl}
    $\|c_v(\pi_v(\g C),\pi_v(C))\|_{_{F_v(\P_v)}}$&$=\|c_v(\g_k h_k C,\g_k C)\|_{_{F_v(\P_v)}}$ \smallskip \\
    $\;$&$=\|c_H(f_{\g_kH}(\g_k h_k C),f_{\g_kH}(\g_k C))\|_{_{F'(\P')}}$ (notice that $\g_k \in R_{\G/H}$) \smallskip \\
    $\;$&$=\|c_H( h_k C, C)\|_{_{F'(\P')}}$ \\
   \end{tabular}
  \end{center} \smallskip 
  Now, for $v=g_1h_1...g_{k-1}h_{k-1}G$, 
  $$ \pi_v(C)=g_1h_1...g_{k-1}h_{k-1} C \text{ and } \pi_v(\g C)=g_1h_1...h_{k-1}g_{k} C.$$ 
  Hence, similarly, we have $\|c_v(\pi_v(\g C),\pi_v(C))\|_{_{F_v(\P_v)}}=\|c_G(g_k C, C)\|_{_{F(\P)}}.$ \smallskip \\
  Thus :
  $$ \|c_V(\g C, C)\|_{_{F_q(\P_{V})}}^q=\sum_{k=1}^n\|c_G(g_k C, C)\|_{_{F(\P)}}^q+\|c_H(h_k C, C)\|_{_{F'(\P')}}^q. $$ \medskip
  
  Let us now consider the structure of space with labelled partitions $(X,\P_T,\ell^q(\P_T))$ induces by the tree structure given by Definition \ref{free_tree}. By Proposition \ref{free_tree_prop}, $\G$ acts by automorphisms on $(X,\P_T,\ell^q(\P_T))$ and we have, for $\g = g_1h_1...g_nh_n c$ a reduced word, 
  $$ \|c_T(\g C, C)\|_{_{\ell^q(\P_{T})}}^q = d_T(\g G, G)= k, $$
  where $2n-2 \leq k \leq 2n$ (depending on the fact that $g_1$ and $h_1$ can possibly be trivial). \medskip \\
  
  Finally, we endow $X$ with a structure of labelled partitions $(X,\P_X,F_X(\P_X))$ given by the pullback of the product structure $(X\times X, \P_V^{\oplus} \sqcup \P_T^{\oplus}, F_q(\P_V)\oplus \ell^q(\P_T))$ via the $\G$-equivariant map $x \mapsto (x,x)$. Hence, we have, for $\g = g_1h_1...g_nh_n c$ and $C \in X_G$ :
  $$\|c_X(\g C,C)\|^q=\sum_{k=1}^n \|c_G(g_k C,C)\|^q_{_{F(\P)}}+\|c_H(h_kC,C)\|^q_{_{F'(\P')}}+d_T(\g G,G),$$ 
  where $c_X$ is the separation map associated with $\P_X$. \smallskip \\

 \end{proof}
 
 \begin{cor}\label{free_act_cor}
  Let $G,H$ be groups, $F$ be a common \emph{finite} subgroup and $q \geq 1$. Assume that $G$ acts \emph{properly} by automorphisms on $(G,\P_G,F_G(\P_G))$ and $H$ acts \emph{properly} by automorphisms on $(H,\P_H,F_H(\P_H))$. \\
  Then there exists a space with labelled partitions $(X,\P_X,F_X(\P_X))$ on which $G*_F H$ acts properly by automorphisms, and morevover $F_X(\P_X)$ is isometrically isomorphic to a closed subspace of $\rexp{q}{\bigoplus} F_G(\P_G)\oplus \rexp{q}{\bigoplus} F_H(\P_H) \oplus \ell^q$.
 \end{cor}
 
 Before we prove this corollary, we need the following lemma :
 
 \begin{lem}\label{finite_labpart}
  Let $G$ be a group and $F$ be a finite subgroup of $G$. Assume that $G$ is endowed with a structure of space with labelled partitions $(G,\P,F(\P))$ on which it acts by automorphisms via left-translation. \\
  Then there exists a structure of space with labelled partitions $(G/F, \P', F'(\P'))$ on which $G$ acts by automorphisms via its natural action on the quotient $G/F$ and where $F'(\P')$ is isometrically isomorphic to a closed subspace of $F(\P)$. Moreover, there exists $K\geq 0$ such that, for all $g,g' \in G$ :
  $$ \|c(g,g')\|_{_{F'(\P')}}+K\geq \|c'(gF,g'F)\|_{_{F'(\P')}} \geq \|c(g,g')\|_{_{F'(\P')}}-K, $$
  where $c,c'$ are the respective separation maps of $(G,\P,F(\P))$ and $(G/F, \P', F'(\P'))$. \\
  In particular, $G\curvearrowright(G,\P,F(\P))$ is proper if, and only if, $G\curvearrowright(G/F, \P', F'(\P'))$ is proper.
 \end{lem}
 
 \begin{proof}
  For $p \in \P$, we define the labelling function $p':G/F \rightarrow \K$, by, for $g \in G$ :
  $$ p'(gF):=\frac{1}{\#F}\sum_{f \in F}p(gf). $$
  Notice that $p'$ is well-defined since $\sum_{f \in F}p(g'f)=\sum_{f \in F}p(gf)$ for every $g' \in gF$. \smallskip \\
  We consider the family of labelling functions $\P'=\{p' \st p\in \P \}$ and we denote by $c'$ its associated separation function. For $p \in \P$, $g,g' \in G$, we have :
  $$ c'(gF,g'F)(p')= \frac{1}{\#F}\sum_{f \in F}c(gf,g'f)(p).$$
  Then, if we set $E'(\P'):=\Vect(c'(gF,g'F)\st g,g' \in G)$, the linear operator $T:E'(\P') \rightarrow F(\P)$ such that :
  $$T :c'(gF,g'F)\mapsto \frac{1}{\#F}\sum_{f \in F}c(gf,g'f),$$
  is injective. Hence we can consider the Banach space $F'(\P')$ defined as the closure of $E'(\P')$ endowed with the norm $\|\k\|_{_{F'(\P')}}:=\|T(\k)\|_{_{F(\P)}}$.
  As $c'(gF,g'F)$ belongs to $F'(\P')$ for all $g,g' \in G$, it follows that $(G/F,\P',F'(\P'))$ is a space with labelled partitions, and it is clear that $G$ acts on it by automorphisms via its natural action on $G/F$. \\
  Now, let us consider $\eta=\frac{1}{\#F}\sum_{f \in F}c(f,e) \in F(\P)$. Then we have, for $g \in G$ :
   \begin{center}

  \begin{tabular}{rl}$\|c'(gF,F)\|_{_{F'(\P')}}$&$=\|\frac{1}{\#F}\sum_{f \in F}c(gf,f)\|_{_{F(\P)}}$ \smallskip \\
                       &$=\|c(g,e)+\eta \circ \phi_g - \eta\|_{_{F(\P)}}.$ \smallskip \\
  \end{tabular}
 \end{center}
 As $ \k \mapsto \k \circ \phi_g$ is an isometry of $F(\P)$, we have, by triangular inequality, $\|\eta \circ \phi_g - \eta\|_{_{F(\P)}} \leq 2\|\eta\|_{_{F(\P)}}$. Hence, again by triangular inequalities : \\
 $$\|c(g,e)\|_{_{F(\P)}}+K \geq \|c'(gF,F)\|_{_{\F'(\P')}} \geq \|c(g,e)\|_{_{F(\P)}}-K,$$
 where $K=2\|\eta\|_{_{F(\P)}}$.
 \end{proof}
 
 \begin{proof}[Proof of Corollary \ref{free_act_cor}]
  Assume that $G$ acts properly by automorphisms on $(G,\P_G,F_G(\P_G))$ and $H$ acts properly by automorphisms on $(H,\P_H,F_H(\P_H))$. Then, by Lemma \ref{finite_labpart}, there exists spaces with labelled partitions $(G/F,\P,F(\P))$ and $(H/F,\P',F'(\P'))$ on which $G$ and $H$ respectively act properly via their natural actions on quotients. Thus we can apply Theorem \ref{free_act_theo} : $G*_F H$ acts by automorphisms on a space with labelled partitions $(X,\P_X,F_q(\P_X))$ where $X$ is the total space of the tree of $F$-cosets spaces associated with $G*_F H$ and $F_q(\P_X)\lesssim \rexp{q}{\bigoplus} F(\P)\oplus \rexp{q}{\bigoplus} F'(\P') \oplus \ell^q$. Moreover, we have, for $\g=g_1h_1...g_nh_n f \in G*_F H$ :
  $$\|c_X(\g F,F)\|^q=\sum_{k=1}^n \|c_G(g_k F,F)\|^q_{_{F(\P)}}+\|c_H(h_kF,F)\|^q_{_{F'(\P')}}+d_T(\g G,G).$$
  For $ R \geq 0$, $\|c_X(\g F,F)\| \leq R$ implies that $2n-2 \leq d_T(\g G,G) \leq R^q$, $\|c_G(g_k F,F)\|_{_{F(\P)}} \leq R$ and $\|c_H(h_k F,F)\|_{_{F'(\P')}} \leq R$. \\
  Hence, for all $R \geq 0$, $\{ \g=g_1h_1...g_nh_n f \st \|c_X(\g F,F)\| \leq R \}$ is a subset of :
  $$ \left\lbrace \g=g_1h_1...g_nh_n f \st n \leq 2(R^q +2) \text{ and } \|c_G(g_k F,F)\|_{_{F(\P)}} \leq R, \; \|c_H(h_k F,F)\|_{_{F'(\P')}} \leq R\right\rbrace, $$
  which is a finite set as $F$ is finite, $n$ is bounded and $G \curvearrowright (G/F,\P,F(\P))$, $H \curvearrowright (H/F,\P',F'(\P'))$ are proper. \\
  Thus, $\G \curvearrowright (X,\P_X,F_q(\P_X))$ is proper.
 \end{proof}
 
 \begin{proof}[Proof of Theorem \ref{amalgamflp}]
  The necessary condition is clear as $G,H$ are subgroups of $G*_F H$. \\
  Now, assume $G,H$ have property $PL^p$. It follows from Corollary \ref{labpart_aflp} that there exists structures of spaces with labelled partitions $(G,\P,F(\P'))$ and $(H,\P',F'(\P'))$ on which $G$ and $H$ respectively, act properly by automorphisms via left-translations. Moreover $F(\P)$ is isometrically isomorphic to a closed subspace of a $L^p$ space and so does $F'(\P')$.\\
  Thus, as $F$ is finite, by Corollary \ref{free_act_cor}, there exists a space with labelled partitions $(X,\P_X,F_X(\P_X))$ on which $G*_F H$ acts properly by automorphisms where :
  $$ F_X(\P_X) \lesssim \rexp{p}{\bigoplus} F(\P)\oplus \rexp{p}{\bigoplus} F'(\P') \oplus \ell^p.$$
  Hence, $ F_X(\P_X)$ is isometrically isomorphic to a closed subspace of a $L^p$ space by Proposition \ref{lpisomisom}. By Corollary \ref{labpart_aflp}, it follows that $G*_F H$ has property $PL^p$.
 \end{proof}

 \nocite{*}
  \bibliographystyle{alpha} 
 \bibliography{biblio}

\end{document}